\documentclass{amsart}

\usepackage{amssymb,amsmath,esint}
\usepackage[all]{xy}
\usepackage{graphicx}
\usepackage{enumerate}

\usepackage{colortbl}
\usepackage{xcolor}

\usepackage{amsrefs}

\colorlet{linkequation}{blue}
\usepackage[colorlinks,plainpages=true,pdfpagelabels,hypertexnames=true,colorlinks=true,pdfstartview=FitV,linkcolor=blue,citecolor=red,urlcolor=black]{hyperref}

\PassOptionsToPackage{unicode}{hyperref}
\PassOptionsToPackage{naturalnames}{hyperref}

\definecolor{dgreen}{rgb}{0,0.5,0}
\definecolor{violet}{rgb}{0.5,0,0.5}
\definecolor{dred}{rgb}{0.7,0,0}
\definecolor{ddred}{rgb}{0.5,0,0}
\definecolor{dblue}{rgb}{0,0,0.5}
\definecolor{ddblue}{rgb}{0,0,0.3}

\newtheorem{theorem}{Theorem}[section]
\newtheorem{lemma}[theorem]{Lemma}

\newtheorem{definition}[theorem]{Definition}

\newtheorem{remark}[theorem]{Remark}

\numberwithin{equation}{section}

\DeclareMathOperator*{\ddiv}{div}

\newcommand{\integral}[3]{\int_{#1} #2 \ #3}

\newcommand{\mint}[3]{\fint_{#1} #2 \ #3}

\newcommand{\norm}[1]{\left| #1\right|}
\newcommand{\Norm}[1]{\left|\hspace{-0.2mm}\left| #1 \right|\hspace{-0.2mm}\right|}
\newcommand{\gh}[1]{\left( #1\right)}
\newcommand{\mgh}[1]{\left\{ #1\right\}}
\newcommand{\bgh}[1]{\left[ #1\right]}

\newcommand{\ww}{W^{1,p(\cdot)}}
\newcommand{\wwz}{W_0^{1,p(\cdot)}}

\newcommand{\OO}{\Omega}

\newcommand{\ga}{p^-}
\newcommand{\gb}{p^+}

\newcommand{\ma}{\mathbf{a}}

\newcommand{\La}{\Lambda}

\newcommand{\M}{\mathcal{M}}

\newcommand{\bb}{\mathbb{R}}
\newcommand{\R}{\mathbb{R}}
\newcommand{\N}{\mathbb{N}}
\newcommand{\A}{\mathcal{A}_{\psi_1,\psi_2}^g}
\newcommand{\diva}{\ddiv \mathbf{a}}

\newcommand{\data}{\mathsf{data}}

\AtBeginDocument{%
   \def\MR#1{}
}

\begin{document}

\title[Elliptic double obstacle problems with measure data]{Nonlinear gradient estimates for elliptic double obstacle problems with measure data}

\author{Sun-Sig Byun}
\address{S.-S. Byun: Department of Mathematical Sciences and Research Institute of Mathematics, Seoul National University, Seoul 08826, Republic of Korea}
\email{byun@snu.ac.kr}

\author{Yumi Cho}
\address{Y. Cho: Department of Mathematical Sciences, Seoul National University,
Seoul 08826, Republic of Korea}
\email{imuy31@snu.ac.kr}

\author{Jung-Tae Park}
\address{J.-T. Park: Korea Institute for Advanced Study, Seoul 02455, Republic of Korea}
\email{ppark00@kias.re.kr}

\thanks{S.-S. Byun was supported by NRF-2017R1A2B2003877.  Y. Cho was supported by NRF-2019R1I1A1A01064053. J.-T. Park was supported by NRF-2019R1C1C1003844.}


\date{May 2, 2020}


\begin{abstract}
We study quasilinear elliptic double obstacle problems with a variable exponent growth when the right-hand side is a measure. A global Calder\'{o}n-Zygmund estimate for the gradient of an approximable solution is obtained in terms of the associated double obstacles and a given measure, identifying minimal requirements for the regularity estimate.
\end{abstract}

\maketitle

\section{Introduction and main results}

We consider a double obstacle problem involving the following elliptic equation with measure data:
\begin{equation}\label{pem1}
-\ddiv \mathbf{a}(Du,x) = \mu \quad\text{in} \hspace{2mm} \OO,
\end{equation}
where $\mu$ is a signed Radon measure on $\Omega$ with finite total mass $|\mu|(\OO)<\infty$. Here $\Omega$ is a bounded domain in $\R^n$ $(n\ge2)$ with nonsmooth boundary $\partial\OO$, and the nonlinearity $\mathbf{a}=\mathbf{a}(\xi,x) : \R^n\times\R^n\rightarrow\R^n$ satisfies a variable $p(x)$-growth condition.
We are interested in finding a suitable notion of a solution $u$ with two-sided constraints on the values of $u$, $\psi_1\le u\le\psi_2$ a.e. in $\Omega$, where the given double obstacles $\psi_1,\psi_2\in W^{1,p(\cdot)}(\Omega)$ satisfy
\begin{equation}\label{psi-g-L1}
\left\{\begin{aligned}
&\mbox{$\psi_1\le\psi_2$ a.e. in $\Omega$,   $\psi_1\le0\le\psi_2$ a.e. on $\partial\Omega$,}\\
&\mbox{$\diva(D\psi_1,\cdot),\ \diva(D\psi_2,\cdot)\in L^1(\Omega)$.}\\
\end{aligned}\right.
\end{equation}

Under possibly optimal conditions on $p(\cdot)$, $\mathbf a$ and $\Omega$, we show that such a solution $u$ satisfies the following gradient bound:
\begin{equation}\label{intro-czest}
\begin{aligned}
\integral{\OO}{|Du|^q}{dx} &\le c \integral{\OO}{\M_1(\mu)^{\frac{q}{p(x)-1}}}{dx}+c\integral{\OO}{\M_1(\Psi_1)^{\frac{q}{p(x)-1}}}{dx}\\
&\qquad +c\integral{\OO}{\M_1(\Psi_2)^{\frac{q}{p(x)-1}}}{dx}+ c
\end{aligned}
\end{equation}
for every $0<q<\infty$. Here $\M_1$ is the {\em fractional maximal functions of order 1} for $\mu$ and $\Psi_i$ defined by
\begin{align}\label{fractional max ftn}
\mathcal{M}_1(\mu)(x):=\sup_{r>0}\frac{r|\mu|(B_r(x))}{|B_r(x)|}
\ \mbox{and} \
\mathcal{M}_1(\Psi_i)(x):=\sup_{r>0}r\fint_{B_r(x)} |\Psi_i(y)| \ dy
\end{align}
for $x\in\R^n$, where $\Psi_i(x) := \diva(D\psi_i,x)$ for $i=1,2$. Here we have assumed that $\mu$ and $\Psi_i$ are defined in $\R^n$ by considering the zero extension to $\R^n\setminus\Omega$ (see \cite{AH96,KS03} for a further discussion on the fractional maximal function).

Specifically, we assume that a function $p(\cdot)$ is a continuous function on $\overline{\OO}$ satisfying
\begin{equation}\label{p-range}
2-\frac{1}{n} < \ga \le p(\cdot) \le \gb < \infty
\end{equation}
for some constants $\ga$ and $\gb$ and that the vector field
$\mathbf{a}=\mathbf{a}(\xi,x) : \bb^n \times \R^n \rightarrow
\bb^n$ is differentiable in $\xi$ for a.e. $x$ and measurable in $x$ for all $\xi$.
In addition, there exist two positive constants $\Lambda_1$ and $\Lambda_2$ such that
\begin{equation}\label{growth}
\left\{
\begin{aligned}
& |\mathbf{a}(\xi,x)| + |\xi||D_{\xi}\mathbf{a}(\xi,x)| \le \Lambda_1 |\xi|^{p(x)-1},\\
& \Lambda_2 |\xi|^{p(x)-2}|\eta|^2 \le D_{\xi}\mathbf{a}(\xi,x)\eta \cdot \eta
\end{aligned}\right.
\end{equation}
for every $\xi \in \bb^n \setminus \{0\}$, $\eta \in \bb^n$, and for a.e. $x \in \bb^n$.
Then the structure condition $\eqref{growth}_2$ yields the following monotonicity
condition:
\begin{equation}\label{monotonicity}
\big(\mathbf{a}(\xi,x)-\mathbf{a}(\eta,x)\big)\cdot (\xi-\eta ) \ge
\left\{
\begin{alignedat}{3}
&\tilde{\Lambda} \left( \norm{\xi}^2 + \norm{\eta}^2 \right)^{\frac{p(x)-2}{2}} \norm{\xi-\eta}^2  &&\quad \mbox{if $1<p(x)<2$},\\
&\tilde{\Lambda} \norm{\xi-\eta}^{p(x)} &&\quad \mbox{if $p(x)\ge2$}
\end{alignedat}\right.
\end{equation}
for every $ \xi, \eta \in \bb^n$, for a.e. $x\in\bb^n$, and for some constant
$\tilde{\Lambda}=\tilde{\Lambda}(n,\Lambda_2,\ga,\gb)>0$.

For the case $p^->n$, the measure $\mu$ belongs to $W^{-1,p'(\cdot)}(\Omega)$, the dual space of $\wwz(\OO)$, by Morrey's inequality. Then the obstacle problem can be characterized as the variational inequality
\begin{align}\label{regular-ob}
\int_\Omega \mathbf{a}(Du,x)\cdot(D\varphi-Du)\ dx\ge\int_\Omega\mu(\varphi-u)\ dx
\end{align}
for all $\varphi\in \mathcal{A}_{\psi_1,\psi_2}$. Here the admissible set $\mathcal{A}_{\psi_1,\psi_2}$ is defined by
\begin{equation}\label{admissible set}
\mathcal{A}_{\psi_1,\psi_2}:=\mgh{\varphi\in W^{1,p(\cdot)}_0(\Omega): \psi_1\le \varphi\le\psi_2 \mbox{ a.e. in $\Omega$}},
\end{equation}
where the obstacles $\psi_1,\psi_2 \in W^{1,p(\cdot)}(\Omega)$ satisfy $\eqref{psi-g-L1}_1$. Note that $\mathcal{A}_{\psi_1,\psi_2} \neq \emptyset$ since $(\psi_1)_+ -(\psi_2)_- \in \mathcal{A}_{\psi_1,\psi_2}$. In this case, there exists a unique weak solution $u \in \mathcal{A}_{\psi_1,\psi_2}$ to the variational inequality \eqref{regular-ob} (see \cite{KS00,Sho97}), and there have been regularity results regarding Calder\'on-Zygmund type estimates such as \eqref{intro-czest} (see for example \cite{BDM11,BCW12,BC15,BCO16,Erh14,BR20}). Thus, we only focus on the case $p^- \le n$ for which obstacle problem with a bounded Radon measure $\mu$ is not in general representable by a variational inequality like \eqref{regular-ob}. Scheven \cite{Sch12a,Sch12b} introduced {\em a limit of approximation solutions} of the obstacle problem related to the problem \eqref{pem1} with a $p$-Laplace type.
We adopt this concept to our problem with the $p(\cdot)$-growth (see Section~\ref{Preliminaries} for notation appearing in the following definition):
\begin{definition}
\label{def-limit-of approx-sol}
Let $\psi_1,\psi_2, g \in W^{1,p(\cdot)}(\Omega)$ with $\psi_1\le g\le\psi_2$ a.e. in $\Omega$ and let $\mu\in\mathcal{M}_b(\Omega)$.
We say that $u\in\mathcal{T}_g^{1,p(\cdot)}(\Omega)$ with $\psi_1\le u\le\psi_2$ a.e. in $\Omega$
is {\em a limit of approximating solutions} to the obstacle problem $OP(\psi_1,\psi_2;\mu)$ if
there exist functions $\mu_i\in W^{-1,p'(\cdot)}(\Omega)\cap L^1(\Omega)$ with
\begin{equation}\label{sol-def-mu}
\left\{
\begin{aligned}
&\mu_i\overset{*}{\rightharpoonup}\mu\quad\mbox{in $\mathcal{M}_b(\Omega)$ as $i\rightarrow\infty$},\\
&\limsup_{i\rightarrow\infty} |\mu_i|(V) \le |\mu|(\overline{V}) \quad \mbox{for any open set $V\subset\R^n$}
\end{aligned}
\right.
\end{equation}
and weak solutions $u_i\in W^{1,p(\cdot)}(\Omega)$ with $ \psi_1\le u_i\le\psi_2$  a.e. in $\Omega$ to the variational inequalities
\begin{equation}\label{var-ineq}
  \int_\Omega \mathbf{a}(Du_i,x)\cdot (D\varphi-Du_i) \ dx
  \ge \int_{\Omega}\mu_i (\varphi-u_i) \ dx
\end{equation}
for all $\varphi\in u_i+W^{1,p(\cdot)}_0(\Omega)$ with $\psi_1\le\varphi\le\psi_2$ a.e. in $\Omega$, such that
\begin{equation}\label{seq-converge}
u_i\rightarrow u \quad in \quad W^{1, r(\cdot)}(\Omega)
\end{equation}
for all continuous functions $r(\cdot)$ on $\overline{\Omega}$ satisfying
$0< r(\cdot)<\min\mgh{\frac{n(p(\cdot)-1)}{n-1}, p(\cdot)}$.
\end{definition}

\begin{remark}
\begin{enumerate}[\rm(i)]
\item 
The given data $g$ is used as a test function in \eqref{var-ineq} (see Lemma~\ref{lem-grad-bd} below).


\item Since $p(\cdot) >2-\frac{1}{n}$ in \eqref{p-range}, it follows from \eqref{seq-converge} that $u_i\rightarrow u$ in $W^{1, 1}(\Omega)$. This property  is necessary to prove the existence of a limit of approximating solution (see Section~\ref{existence} below). There have been many existence results for obstacle problems with measure data. For instance, see \cite{DD99, DL02} for linear problems, \cite{BS02,BP05} for semilinear problems, and \cite{BC99,Leo01,Leo00,OR01a,OR01b,Sch12a} for $p$-Laplace type problems. We also refer to \cite[Section 1.1]{Sch12a} for further discussions in the literature. On the other hand, the uniqueness of a limit of approximating solution remains open except for the linear problems, even in cases without obstacles (see \cite{DMOP99}).
\end{enumerate}
\end{remark}

We next present our hypotheses on $p(\cdot)$, $\mathbf a$ and $\Omega$.
\begin{definition}
Let $R>0$ and $0<\delta \le\frac{1}{8}$.
We say $(p(\cdot),\mathbf{a},\OO)$ is {\em $(\delta,R)$-vanishing}
if the followings hold:
\begin{enumerate}[\rm(i)]
\item $p(\cdot)$ admits $\omega : [0,\infty) \to [0,\infty)$ as a modulus of continuity satisfying
\begin{equation}\label{log-holder}
\sup_{0<r \le R} \omega(r) \log\gh{\frac{1}{r}} \le \delta.
\end{equation}

\item $\mathbf{a}$ satisfies
\begin{equation}\label{small BMO}
\sup_{0<r \le R} \sup_{y\in\bb^n} \mint{B_r(y)}{\theta\gh{\mathbf{a},B_r(y)}(x)}{dx} \le \delta,
\end{equation}
where
\begin{equation*}\label{AA}
\theta\gh{\mathbf{a},D}(x) := \sup_{\xi \in \bb^n \setminus \{0\}} \left|\frac{\mathbf{a}(\xi,x)}{|\xi|^{p(x)-1}} - \overline{\gh{\frac{\mathbf{a}(\xi,\cdot)}{|\xi|^{p(\cdot)-1}}}}_D\right|
\end{equation*}
for a bounded open set $D \subset \bb^n$.

\item $\OO$ is {\em $(\delta,R)$-Reifenberg flat}, namely,
for any $x_0 \in \partial\OO$ and for any $0<r \le R$, there exists
a coordinate system $\{y_1,\cdots,y_n\}$ with the origin at $x_0$ such that
\begin{equation*}
\label{Reifenberg} \{y\in B_r: y_n > \delta r\}
\subset B_r  \cap \OO
\subset  \{y\in B_r: y_n > - \delta r\}.
\end{equation*}
\end{enumerate}
\end{definition}

\begin{remark}
\begin{enumerate}[\rm(i)]
\item A function $p(\cdot)$ satisfying \eqref{log-holder} is log-H\"{o}lder continuous, which is more general than H\"{o}lder continuous (see Section \ref{Preliminaries} for details).

\item If a vector field $\mathbf{a}$ satisfies \eqref{small BMO}, then $x \mapsto \frac{\mathbf{a}(\xi,x)}{|\xi|^{p(x)-1}}$ is of small BMO for each $\xi$ uniformly. This property plays an important role in our perturbation argument in Section \ref{comparison} for the regularity theory (see also for instance \cite{BW04,BCW12,BC15,BCO16,BOR16,BOP17, Phu14a, Ngu15,NP19}).

\item If $\OO$ is $(\delta,R)$-Reifenberg flat, then we have the following measure density properties:
\begin{equation*}
\left\{
\begin{aligned}
& \sup_{x \in \OO} \sup_{0<r \le R} \frac{|B_r(x)|}{|\OO \cap B_r(x)|} \le \gh{\frac{2}{1-\delta}}^n \le \gh{\frac{16}{7}}^n,\\
& \inf_{x \in \partial\OO} \inf_{0<r \le R} \frac{|\OO^c \cap B_r(x)|}{|B_r(x)|} \ge \gh{\frac{1-\delta}{2}}^n \ge \gh{\frac{7}{16}}^n.
\end{aligned}
\right.
\end{equation*}
We refer
to \cite{BW04, LMS14, Tor97} for further discussions on Reifenberg flat domains.
\end{enumerate}
\end{remark}

For simplicity of notation, we employ the word ``$\data$'' to denote any structure constant that depends only on $n$, $\La_1$, $\La_2$, $\ga$ and $\gb$. Moreover we assume that
$g\in W^{1,p(\cdot)}_0(\Omega)$ satisfies
\begin{equation}\label{g-L1}
\mbox{$\psi_1\le g\le\psi_2$ a.e. in $\Omega$ \ and \ $\diva(Dg,\cdot)\in L^1(\Omega)$.}
\end{equation}

Now, we are ready to state our main result.
\begin{theorem}\label{main theorem}
Let $0<q<\infty$, $\ga \le n$ and $\mu\in\mathcal{M}_b(\Omega)$,
and let \eqref{p-range} and \eqref{growth} hold.
Suppose that
$\psi_1, \psi_2\in W^{1,p(\cdot)}(\Omega)$
and
$g\in W^{1,p(\cdot)}_0(\Omega)$ satisfy \eqref{psi-g-L1} and \eqref{g-L1}, respectively.
Then there exists a small constant $\delta=\delta(\data, q) > 0$
such that
if
$(p(\cdot),\mathbf{a},\Omega)$ is $(\delta,R)$-vanishing for some $0<R<1$,
then
for any limit of approximation solutions $u\in\mathcal{T}_g^{1,p(\cdot)}(\Omega)$ to the obstacle problem $OP(\psi_1,\psi_2;\mu)$,
there is a constant $c =c(\data,\omega(\cdot),q,R,\OO,\alpha)>0$
such that
\begin{equation}
\begin{aligned}\label{main_r}
\integral{\OO}{|Du|^q}{dx} &\le c \gh{V+ V^{\frac{1}{(\ga-1)(1-\alpha)}}}^{(n+1)q} + c\integral{\OO}{\M_1(\mu)^{\frac{q}{p(x)-1}}}{dx}\\
&\qquad+c\integral{\OO}{\M_1(\Psi_1)^{\frac{q}{p(x)-1}}}{dx}
+c\integral{\OO}{\M_1(\Psi_2)^{\frac{q}{p(x)-1}}}{dx}+ c
\end{aligned}
\end{equation}
for every constant $\alpha$ with $0< \alpha \le \frac{1}{2}\left(\frac{n}{n-1} -
\frac{1}{\ga-1} \right) <1$, where
\begin{equation}\label{VV}
\begin{aligned}
V &:=
|\mu|(\OO)
+
\int_\Omega|\diva(D\psi_1,x)|\ dx
+
\int_\Omega|\diva(D\psi_2,x)|\ dx\\
&\qquad+
\int_\Omega|\diva(Dg,x)|\ dx
+
\int_\Omega|Dg|\ dx .
\end{aligned}
\end{equation}
\end{theorem}

\begin{remark}\label{main-rk}
The assumptions $\diva(D\psi_1,\cdot)$, $\diva(D\psi_2,\cdot)$, $\diva(Dg,\cdot)\in L^1(\Omega)$ in Theorem \ref{main theorem} come from the monotonicity condition \eqref{monotonicity}.
If $p^- \ge 2$, we can indeed drop these assumptions (see Remark~\ref{rem-p-const-2}, Remark~\ref{rem-p-const-3} and Section~\ref{cover} below for details). In the case $p^- \ge 2$, we obtain, instead of \eqref{main_r},
\begin{equation*}
\begin{aligned}
\integral{\OO}{|Du|^q}{dx} &\le c \gh{W+ W^{\frac{1}{(\ga-1)(1-\alpha)}}}^{(n+1)q} + c\integral{\OO}{\M_1(\mu)^{\frac{q}{p(x)-1}}}{dx}\\
&\qquad+c\integral{\OO}{\M\gh{|D\psi_1|^{p(\cdot)}}^{\frac{q}{p(x)}}}{dx}
+c\integral{\OO}{\M\gh{|D\psi_2|^{p(\cdot)}}^{\frac{q}{p(x)}}}{dx}+ c,
\end{aligned}
\end{equation*}
where $\M$ is defined by \eqref{max ftn} and
\begin{equation*}
\begin{aligned}
W &:=
|\mu|(\OO)
+
\int_\Omega|D\psi_1|^{p(x)}\ dx
+
\int_\Omega|D\psi_2|^{p(x)}\ dx
+
\int_\Omega|Dg|^{p(x)}\ dx.
\end{aligned}
\end{equation*}
\end{remark}

\begin{remark}\label{main-rk2}
\begin{enumerate}[\rm(i)]
\item The both constants $c$ in Theorem~\ref{main theorem} and Remark~\ref{main-rk} blow up when $\alpha \searrow 0$ (see Remark~\ref{rem-p-const-3.5} below).

\item The occurrence of terms $V, W$ results from the absence of normalization property (roughly speaking, a constant multiple of solution becomes another solution) of our problem with a variable exponent growth $p(\cdot)$. In obstacle problems with a constant growth $p(\cdot)\equiv p$, however, we can employ the normalization property to obtain a more natural estimate than \eqref{main_r} (see Remark~\ref{main rk3} below).
\end{enumerate}
\end{remark}

In \cite{Sch12a,Sch12b}, Scheven proved some potential estimates for measure data problems with a one-sided obstacle and a constant $p$-growth. On the other hand, we are dealing with two-sided obstacles and a variable exponent $p(\cdot)$-growth, which poses many difficulties (see for example Remark~\ref{main-rk2} above) in obtaining the desired estimate \eqref{main_r}. To handle these difficulties, we revisit the maximal function approach in \cite{CP98,Min10,AM05} alongside with $L^1$-comparison estimates (see Section~\ref{comparison} below), a standard energy $L^1$-estimate (see Section~\ref{Standard $L^1$-estimate} below) and a Vitali type covering lemma (see Lemma~\ref{covering} and Section~\ref{Vitali type covering} below). In addition, we obtain \eqref{main_r} on a nonsmooth domain beyond the Lipschitz category. 

We organize this paper as follow. In Section~\ref{Preliminaries}, we introduce standard notation, some function spaces and auxiliary results. Section~\ref{existence} provides the existence of a limit of approximating solutions of the obstacle problem $OP(\psi_1,\psi_2;\mu)$, previously introduced in Definition~\ref{def-limit-of approx-sol}. In Section~\ref{comparison}, we deduce comparison estimates between our obstacle problem and its reference problems. In Section~\ref{cover}, we develop a standard energy $L^1$-estimate, and then we finally complete the proof of Theorem~\ref{main theorem} by using a Vitali type covering lemma.

\section{Preliminaries}\label{Preliminaries}
We start with notation which is used throughout the paper.
Let us denote by $B_r(x)$ the open ball in $\bb^n$ with center $x$ and radius $r>0$. Let $B_r := B_r(0)$, $B_r^+ := B_r\cap \{x \in \bb^n : x_n >0 \}$, $\Omega_r(x) := \Omega\cap B_r(x)$, and $\Omega_r := \Omega\cap B_r.$
For $f \in L_{loc}^1(\bb^n)$, we denote
the integral average of $f$ over a bounded open set $D
\subset \bb^n$ by $\overline{(f)}_{D}$, that is,
$$\overline{(f)}_{D} := \mint{D}{f(x)}{dx} = \frac{1}{|D|} \integral{D}{f(x)}{dx}.$$
For a given real-valued function $f$, we set
$$(f)_+ := \max\mgh{f,0} \quad \text{and} \quad (f)_- := -\min\mgh{f,0}.$$
We denote by $\mathcal{M}_b(\Omega)$  a collection of signed Radon measures on $\Omega$ with $|\mu|(\OO)<\infty$ and
define the truncation operator with $k>0$
\begin{equation}\label{truncation}
T_k(y):=
\left\{\begin{alignedat}{2}
&y &&\quad \mbox{if $|y|\le k$},\\
&k\hspace{0.05cm} {\rm{sgn}}(y) &&\quad \mbox{if $|y|> k$}
\end{alignedat}\right.
\end{equation}
for all $y\in\R$.
For a given function $g\in W^{1,p(\cdot)}(\Omega)$,
a function space $\mathcal{T}_g^{1,p(\cdot)}(\Omega)$
consists of all measurable functions $f:\Omega\rightarrow\R$ such that
$T_k(f-g)\in W^{1,p(\cdot)}_0(\Omega)$ for all $k>0$.

Next, we briefly review Lebesgue and Sobolev spaces with variable exponents.
Let $p(\cdot):\R^n\rightarrow (1,\infty)$ be a continuous function with \eqref{p-range}.
We define the {\em variable exponent Lebesgue space} $L^{p(\cdot)}(\OO)$ as the collection of all measurable functions $f$ on $\Omega$ with
$$\rho_{p(\cdot)}(f) :=\integral{\OO}{|f(x)|^{p(x)}}{dx} < \infty,$$
equipped with the Luxemburg norm
$$\Norm{f}_{L^{p(\cdot)}(\OO)} := \inf \mgh{\theta > 0 :
\integral{\OO}{\left|\frac{f(x)}{\theta}\right|^{p(x)}}{dx} \le 1}.$$
Then there is the following relation between the Luxemburg norm and the integral version:
\begin{equation}\label{norm-int relation}
\min\mgh{\rho_{p(\cdot)}(f)^{\frac{1}{\ga}}, \rho_{p(\cdot)}(f)^{\frac{1}{\gb}}} \le \|f\|_{L^{p(\cdot)}(\OO)} \hspace{-0.6mm} \le \max\mgh{\rho_{p(\cdot)}(f)^{\frac{1}{\ga}}, \rho_{p(\cdot)}(f)^{\frac{1}{\gb}}}.
\end{equation}
Moreover, the {\em variable exponent Sobolev space} $\ww(\OO)$ consists of all functions $f\in W^{1,1}(\Omega)\cap L^{p(\cdot)}(\Omega)$ whose first order derivatives belong to $L^{p(\cdot)}(\Omega)$,
and the norm is defined by
$$\Norm{f}_{\ww(\OO)} := \Norm{f}_{L^{p(\cdot)}(\OO)} +
\Norm{|Df|}_{L^{p(\cdot)}(\OO)}.$$
We denote the
closure of $C_c^{\infty}(\OO)$ in $\ww(\OO)$ by $\wwz(\OO)$, and also
 the dual space of
$W^{1,p(\cdot)}_0(\Omega)$ by $W^{-1,p'(\cdot)}(\Omega)$, where $p'(\cdot):=\frac{p(\cdot)}{p(\cdot)-1}$.
Note that they are all separable reflexive Banach spaces.

We now give a crucial condition on variable exponents $p(\cdot)$ to have some important properties for $L^{p(\cdot)}(\OO)$ and $W^{1,p(\cdot)}(\OO)$.
We say that
$p(\cdot)$ is {\em log-H\"{o}lder continuous} in $\OO$ if there
is a constant $c_l>0$ such that
\begin{equation*}\label{log}
|p(x) - p(y)| \le \frac{c_l}{-\log|x-y|}
\end{equation*}
for all $x, y \in \OO$ with $|x-y|
\le \frac{1}{2}$.
This is equivalent to the existence of a nondecreasing concave function $\omega : [0,\infty) \to
[0,\infty)$ with $\omega(0)=0$ satisfying
\begin{equation*}\label{log1}
|p(x)-p(y)| \le \omega\gh{|x-y|} \quad \text{for all} \ x,y \in \OO,
\end{equation*}
and
\begin{equation*}\label{log2}
\sup_{0<r \le \frac{1}{2}} \omega(r) \log\gh{\frac{1}{r}} \le \tilde{c_l}
\end{equation*}
for some constant $\tilde{c_l}>0$.
If $p(\cdot)$ is log-H\"older continuous, then the variable exponent function spaces introduced above have important properties such as the Sobolev embedding theorem and Poincar\'e's inequality.
For further details on the variable exponent function spaces, we refer to the monographs \cite{DHHR11, CF13}.

We next introduce the {\em Hardy-Littlewood maximal function} as an important tool for the proof of our main result (see Section~\ref{cover} below), defined by
\begin{equation}\label{max ftn}
\M f(x) := \sup_{r>0} \mint{B_r(x)}{|f(y)|}{dy} \quad (f \in L_{loc}^1(\bb^n)).
\end{equation}
If $f$ is defined on a bounded open set $D \subset \bb^n$, then we write
$$\M_{D}f := \M (\chi_D f),$$
where $\chi_D$ is the characteristic function over $D$.
For the sake of simplicity,
we drop the index $D$ when $D = \OO$.
We shall use the following well-known estimates:
\begin{equation*}\label{weak 1-1}
\norm{\mgh{x \in \OO : \M f(x) > \theta}} \le \frac{c(n)}{\theta}
\integral{\OO}{|f|}{dx}\ \ \text{ for every }\theta>0,
\end{equation*}
and for $1<p\le\infty$
\begin{equation*}\label{strong p-p}
\Norm{\M f}_{L^p(\OO)} \le c(n,p) \Norm{f}_{L^p(\OO)}.
\end{equation*}

The following standard measure theoretical property will be used in Section~\ref{proof of main theorem} later:
\begin{lemma}[{\cite[Lemma 7.3]{CC95}}]
\label{distribution2} Let $f$ be a measurable
function on a bounded open set $\OO \subset \bb^n$. Let $\lambda>0$
and $m>1$. Then, for any $0<q<\infty$,
\begin{equation*}\label{distri1}
f \in L^q(\OO) \  \iff \  S := \sum_{k\ge 1} m^{qk} \norm{\mgh{x \in \OO : |f(x)| > \lambda m^{k}}} < \infty,
\end{equation*}
and
\begin{equation*}\label{distri2}
\frac1{c}\lambda^q S \le \integral{\OO}{|f|^q}{dx} \le c\lambda^q \gh{|\OO|+S},
\end{equation*}
for some $c=c(m,q)>0$.
\end{lemma}

We end this section with a Vitali type covering lemma as follows:
\begin{lemma}[{\cite[Theorem 2.8]{BW04}}] \label{covering}
Let $0<\varepsilon<1$, $0<\delta\le\frac1{8}$, and $R>0$.
We suppose that
$\OO\subset\R^n$ is
$(\delta,R)$-Reifenberg flat.
For measurable sets $C$ and $D$ with $C \subset D\subset \OO$
and for $0< R_0 \le R$,
the two followings hold:
\begin{enumerate}[\rm(i)]
\item $|C| \le \gh{\frac{1}{1000}}^n \varepsilon |B_{R_0}|$, and

\item if $|C \cap B_{r}(x_o)| \ge \varepsilon |B_{r}(x_o)|$ for $x_o \in \OO$ and $r \in \big(0, \frac{R_0}{1000}\big]$, then
$B_{r}(x_o) \cap \OO \subset D$.
\end{enumerate}
Then we have
$$|C| \le \varepsilon \gh{\frac{10}{1-\delta}}^n  |D| \le
\varepsilon\gh{\frac{80}{7}}^n |D|.$$
\end{lemma}

\section{Existence for a limit of approximating solutions}\label{existence}

In this section we derive the existence of a limit of approximating solutions (introduced in Definition~\ref{def-limit-of approx-sol}) of the obstacle problem $OP(\psi_1,\psi_2;\mu)$.
We first recall the truncation function $T_k$ in \eqref{truncation} and introduce another truncation function
\begin{equation}\label{trun}
\Phi_k(t) := T_1 \left(t-T_k(t)\right) \hspace{2mm} \text{for}
\hspace{1mm} t \in \bb.
\end{equation}

With $\mu_i \in W^{-1,p'(\cdot)}(\Omega)\cap L^1(\Omega)$ under the assumption
\begin{equation}\label{mu_i-bound}
K:=\sup_{i\in\N}\|\mu_i\|_{L^1(\Omega)} <\infty,
\end{equation}
we consider the weak solution $u_i\in\A$ to the variational inequality
\begin{equation}\label{u_i-op}
\int_{\Omega} \ma(x,Du_i)\cdot(D\varphi-Du_i) \ dx \ge \int_{\Omega}\mu_i(\varphi-u_i) \ dx \quad\mbox{for all $\varphi\in\A$,}
\end{equation}
where $g\in W^{1,p(\cdot)}(\Omega)$ is given with $\psi_1\le g\le\psi_2$ a.e. in $\Omega$ and
$$\A:=\mgh{\varphi\in g+W_0^{1,p(\cdot)}(\Omega):\psi_1\le \varphi\le\psi_2 \mbox{ a.e. in $\Omega$} }.$$
We shall show that $u_i \rightarrow u$  in the sense of \eqref{seq-converge} for some $u\in\mathcal{T}_g^{1,p(\cdot)}(\Omega)$ with $\psi_1\le u\le\psi_2$ a.e. in $\Omega$. To see this, we first introduce the following technical lemma (see \cite[Lemma 2.1]{BW09} and \cite[Lemma 4.1]{BH10}):
\begin{lemma}\label{lem-tech-p(x)}
Let $f\in W^{1,p(\cdot)}_0(\Omega)$.
Suppose that there is a constant $c_o$ such that
\begin{align*}
\int_{\{k\le|f|\le k+1\}} |Df|^{p(x)}\ dx \le c_o \quad\mbox{for all $k>0$.}
\end{align*}
Then there exists a constant $c=c(c_o)>0$ such that
$$\|f\|_{W^{1,r(\cdot)}_0(\Omega)}\le c$$
for every continuous function $r(\cdot)$ on $\overline{\Omega}$ satisfying
\begin{equation}\label{r-range}
1\le r(\cdot)<\min\mgh{\frac{n(p(\cdot)-1)}{n-1}, p(\cdot)}.
\end{equation}
\end{lemma}

Now we derive some uniform bounds for the variational inequality \eqref{u_i-op} in the variable exponent setting.
For the constant exponent case $p(\cdot)\equiv p$,  we refer to \cite[Lemma 3.3]{Sch12a} (see also \cite{BG89,BBGGPV95}).
\begin{lemma}\label{lem-grad-bd}
Let $u_i\in\A$ be the weak solution to the variational inequality \eqref{u_i-op} with \eqref{mu_i-bound}.
Then there exists a constant $c_1$ depending only on $\data$ such that
\begin{equation}\label{DT_k}
\int_\Omega \big|D\big[T_k\big(u_i-g\big)\big]\big|^{p(x)} \ dx \le c_1 kK+c_1\int_\Omega |Dg|^{p(x)} \ dx
\end{equation}
for any $k>0$.
Moreover, we have
\begin{equation}\label{Du-bound}
\|u_i-g\|_{W^{1,r(\cdot)}_0(\Omega)} \le c_2
\end{equation}
for some $c_2=c_2(\data, K,\|Dg\|_{L^{p(\cdot)}(\Omega)})>0$ and for any continuous function $r(\cdot)$ on $\overline\Omega$ with \eqref{r-range}.
\end{lemma}
\begin{proof}
We first define for $k>0$
$$ w_i:=u_i+T_k(g-u_i) \quad\mbox{and}\quad  v_i: = u_i + \Phi_k \big(g-u_i \big).$$
Then we have $w_i, v_i\in\A$, since $w_i, v_i\in g+W_0^{1,p(\cdot)}(\Omega)$ and
$$\psi_1\le\min\{u_i,g\}\le w_i,v_i\le\max\{u_i,g\}\le\psi_2 \ \ \text{a.e. in}\  \Omega.$$
Now we take $\varphi=w_i$ in \eqref{u_i-op}, to have
\begin{equation*}
\int_\Omega \mathbf{a}(Du_i,x) \cdot D\big[T_k \big( u_i-g \big)\big] \ dx
\le
\int_\Omega \mu_i T_k \big(u_i - g \big)\ dx.
\end{equation*}
We employ \eqref{mu_i-bound}, \eqref{growth} and Young's inequality to discover
\begin{equation}
\begin{aligned}\label{a-a-1}
&\int_{\{|u_i -g |<k\}} (\mathbf{a}(Du_i,x)-\mathbf{a}(Dg,x)) \cdot(Du_i-Dg) \ dx\\
&\quad\le
k\int_\Omega | \mu_i |\ dx
+ \int_{\{|u_i -g |<k\}}|\mathbf{a}(Dg, x)| |Du_i-Dg| \ dx \\
&\quad\le
kK + c(\varepsilon)\int_\Omega |Dg|^{p(x)}\ dx + \varepsilon\int_{\{|u_i -g |<k\}} |Du_i-Dg|^{p(x)}\ dx
\end{aligned}
\end{equation}
for any $\varepsilon>0$.
We now set
$$\Omega^+:=\{x\in\Omega: p(x)\ge2 \ \mbox{and} \ |u_i-g|<k\}$$
and
$$\Omega^-:=\{x\in\Omega: p(x)<2 \ \mbox{and} \ |u_i-g|<k\}.$$
Then it follows from \eqref{monotonicity} that
\begin{align}\label{3-p>2}
\int_{\Omega^+} |Du_i-Dg|^{p(x)}\ dx
\le
c\int_{\Omega^+} (\mathbf{a}(Du_i,x)-\mathbf{a}(Dg,x)) \cdot(Du_i-Dg) \ dx
\end{align}
for some $c=c(\data)>0$.
In addition, \eqref{monotonicity} and Young's inequality yield
\begin{equation}\label{3-p<2}
\begin{aligned}
&\int_{\Omega^-} |Du_i-Dg|^{p(x)}\ dx\\
&\quad\le
c\int_{\Omega^-} (\mathbf{a}(Du_i,x)-\mathbf{a}(Dg,x) )\cdot(Du_i-Dg) \ dx
+
c\int_\Omega |Dg|^{p(x)}\ dx
\end{aligned}
\end{equation}
for some $c=c(\data)>0$.
Finally we combine \eqref{a-a-1}--\eqref{3-p<2} and take $\varepsilon$ with $c\varepsilon\le\frac12$, to obtain \eqref{DT_k}.

Next, by taking $\varphi=v_i$ in \eqref{u_i-op}, we have
\begin{equation*}
\int_\Omega \mathbf{a}(Du_i,x) \cdot D\big[\Phi_k \big( u_i-g \big)\big] \ dx
\le
\int_\Omega \mu_i \Phi_k \big(u_i - g \big)\ dx.
\end{equation*}
Proceeding analogously to the proof of \eqref{DT_k}, we deduce
\begin{align*}
\int_{\{k<|u_i -g |<k+1\}} |Du_i-Dg|^{p(x)}\ dx
\le
cK + c\int_\Omega |Dg|^{p(x)}\ dx
\end{align*}
for all $k>0$, where $c$ is a positive constant depending only on $\data$.
Applying Lemma~\ref{lem-tech-p(x)}, we obtain \eqref{Du-bound} as desired.
\end{proof}

Finally, we obtain the existence of a limit of approximating solutions of the obstacle problem $OP(\psi_1,\psi_2;\mu)$.
\begin{lemma}\label{lem-exist}
Let $u_i\in\A$ be the weak solution to the variational inequality \eqref{u_i-op} satisfying \eqref{mu_i-bound}.
Then there exist a subsequence $\{i_j\}\subset\N$ and a function $u\in\mathcal{T}^{1,p(\cdot)}_g(\Omega)$ with $\psi_1\le u\le\psi_2$ a.e. in $\Omega$ such that
$$u_{i_j}\rightarrow u \quad in \quad W^{1,r(\cdot)}(\Omega)$$
for all continuous functions $r(\cdot)$ on $\overline{\Omega}$ with \eqref{r-range}.
\end{lemma}
\begin{proof}
The idea of the proof follows from \cite[Lemma 3.4]{Sch12a}. For the sake of completeness, we give the proof.
In view of \eqref{DT_k} and \eqref{norm-int relation}, for any fixed $k>0$, the sequence $\mgh{T_k(u_i-g)}_{i\in\N}$ is uniformly bounded in $W^{1,p(\cdot)}_0(\Omega)$.
Since $W^{1,p(\cdot)}_0(\Omega)\hookrightarrow L^{p(\cdot)}(\Omega)$ is continuous compact embedding, we can assume that
\begin{equation}\label{Tu}
\mbox{$\mgh{T_k(u_i-g)}_{i\in\N}$ is a Cauchy sequence in $L^{p(\cdot)}(\Omega)$}
\end{equation}
for all $k>0$.
Next, according to \eqref{Du-bound}, there are a subsequence of $\{u_i\}_{i\in\N}$, still denoted by $\{u_i\}_{i\in\N}$, and a function $u\in g+W^{1,r(\cdot)}_0(\Omega)$ such that
\begin{equation}\label{seq}
\left\{\begin{alignedat}{2}
u_i\rightarrow u &\quad \mbox{a.e. in $\Omega$},\\
u_i\rightarrow u &\quad \mbox{in $L^{r(\cdot)}(\Omega)$}.
\end{alignedat}\right.
\end{equation}
This implies that $\psi_1\le u\le\psi_2$ a.e. in $\Omega$, and that $T_k(u_i-g)\rightarrow T_k(u-g)$ a.e. in $\Omega$ for all $k>0$.
Then we deduce
\begin{equation*}
T_k(u_i-g)\rightarrow T_k(u-g) \quad \mbox{in $L^{p(\cdot)}(\Omega)$ as $i\rightarrow\infty$}
\end{equation*}
for all $k>0.$
Again using \eqref{DT_k}, we can also assume that
\begin{equation}\label{T_k(u-g)}
T_k(u_i-g)\rightharpoonup T_k(u-g) \quad \mbox{in $W^{1,p(\cdot)}_0(\Omega)$ as $i\rightarrow\infty$},
\end{equation}
and so in particular $u\in\mathcal{T}^{1,p(\cdot)}_g(\Omega)$.

Now we choose any continuous function $r(\cdot)$ on $\overline{\Omega}$ with \eqref{r-range}
and write for fixed $i,j\in\N$,
\begin{equation}
\begin{aligned}\label{I+II}
\int_\Omega |Du_i-Du_j|^{r(x)}\ dx
&=
\int_{\{|u_i-u_j|>k\}}|Du_i-Du_j|^{r(x)}\ dx \\
&\qquad+
\int_{\{|u_i-u_j|\le k\}}|Du_i-Du_j|^{r(x)}\ dx\\
&=: I_{ij}^{(k)} + II_{ij}^{(k)}.
\end{aligned}
\end{equation}
Taking a continuous function $\hat{r}(\cdot)$ such that $r(\cdot)<\hat{r}(\cdot)<\min\mgh{\frac{n(p(\cdot)-1)}{n-1}, p(\cdot)}$,
we discover
\begin{equation}
\begin{aligned}\label{I-est}
I_{ij}^{(k)}
&\le
\int_{\{|u_i-u_j|>k\}}\varepsilon\gh{|Du_i|+|Du_j|}^{\hat{r}(x)} + c(\varepsilon)\ dx\\
&\le
\varepsilon c(n,\ga,\gb)\gh{\int_{\Omega}|Du_i|^{\hat{r}(x)}\ dx
+
\int_{\Omega}|Du_j|^{\hat{r}(x)}\ dx} \\
&\qquad+
c(\varepsilon)|\{x \in\OO : |u_i-u_j|>k\}|
\end{aligned}
\end{equation}
for any $\varepsilon>0.$
Here the two integrals on the right-hand side above are bounded independently from $i,j\in\N$ by \eqref{Du-bound}. Also, \eqref{seq} implies that $\{u_i\}$ is a Cauchy sequence in $L^1(\Omega)$. Therefore
$$
\lim_{i,j\rightarrow\infty} |\{x\in\Omega:|u_i-u_j|>k\}|
\le
\lim_{i,j\rightarrow\infty} \frac1k\int_{\Omega}|u_i-u_j|\ dx =0
\quad\mbox{for every $k>0$}.
$$
This and \eqref{I-est} yield
\begin{equation}\label{I-est-1}
\lim_{i,j\rightarrow\infty}I_{ij}^{(k)}=0\quad\mbox{for any $k>0$}.
\end{equation}
Next, for the estimate $II_{ij}^{(k)}$ we write
\begin{align}\label{II-est}
II_{ij}^{(k)}
=
\int_{\Omega_{ij}^+} |Du_i-Du_j|^{r(x)}\ dx
+
\int_{\Omega_{ij}^-} |Du_i-Du_j|^{r(x)}\ dx,
\end{align}
where
$$\Omega_{ij}^+ :=\mgh{x\in\Omega:|u_i-u_j|\le k \ \text{and} \ p(x)\ge2}$$
and
$$\Omega_{ij}^- :=\mgh{x\in\Omega:|u_i-u_j|\le k \ \text{and} \ p(x)<2}.$$
Now we use a comparison function $\varphi=u_i+T_k(u_j-u_i)$ in the variational inequality \eqref{u_i-op} for $u_i$, while we choose $\varphi=u_j+T_k(u_i-u_j)$ in the version for $u_j$. Then we infer that for any $k>0$
\begin{equation*}
\int_\Omega ( \mathbf{a}(Du_i,x) - \mathbf{a}(Du_j,x) )\cdot D\big[ T_k(u_i-u_j) \big]\ dx
\le
\int_\Omega (\mu_i-\mu_j)T_k(u_i-u_j)\ dx.
\end{equation*}
It follows from \eqref{monotonicity} and \eqref{mu_i-bound} that
\begin{equation}\label{mono-est}
\int_{\Omega_{ij}^+} |Du_i-Du_j|^{p(x)}\ dx \le 2kK
\end{equation}
and
\begin{equation}\label{mono-est-1}
\int_{\Omega_{ij}^-} (|Du_i|^2+|Du_j|^2)^{\frac{p(x)-2}2}|Du_i-Du_j|^2 \ dx \le 2kK.
\end{equation}
Keeping in mind $r(\cdot)<p(\cdot)$, the first integral on the right-hand side in \eqref{II-est} yields from Young's inequality and \eqref{mono-est} that for any $\varepsilon>0$
\begin{equation}
\begin{aligned}\label{II-1}
\int_{\Omega_{ij}^+} |Du_i-Du_j|^{r(x)}\ dx
&\le
c(\varepsilon)\int_{\Omega_{ij}^+} |Du_i-Du_j|^{p(x)}\ dx + \varepsilon|\Omega|\\
&\le
c(\varepsilon)kK + \varepsilon|\Omega|.
\end{aligned}
\end{equation}
For the second integral on the right-hand side in \eqref{II-est},
we employ \eqref{mono-est-1} and Young's inequality to have
\begin{equation*}
\begin{aligned}
\int_{\Omega_{ij}^-}|Du_i-Du_j|^{r(x)}\ dx
&\le
c(\varepsilon)\int_{\Omega_{ij}^-} \gh{|Du_i|^2+|Du_j|^2}^{\frac{p(x)-2}2}|Du_i-Du_j|^2\ dx\\
&\qquad+
\varepsilon\int_{\Omega_{ij}^-} \gh{|Du_i|^2+|Du_j|^2}^{\frac{r(x)(2-p(x))}{2(2-r(x))}}\ dx\\
&\le
c(\varepsilon)kK
+
c\varepsilon\int_\Omega \gh{|Du_i|^{r(x)}+|Du_j|^{r(x)}} \ dx + \varepsilon|\Omega|
\end{aligned}
\end{equation*}
for any $\varepsilon>0.$
In the last inequality we used that $\frac{r(x)(2-p(x))}{2-r(x)}<r(x)$ for $x \in \Omega_{ij}^-$.
Then it follows from \eqref{Du-bound} that
\begin{equation}\label{II2}
\int_{\Omega_{ij}^-}|Du_i-Du_j|^{r(x)}\ dx
\le
c(\varepsilon)kK+c\varepsilon
\end{equation}
for some $c=c(\data,K,\|Dg\|_{L^{p(\cdot)}(\Omega)}, |\Omega|)>0$.
Inserting  \eqref{II-1} and \eqref{II2} into \eqref{II-est}, we arrive at
\begin{equation*}
II_{ij}^{(k)}
\le
c(\varepsilon)kK+c\varepsilon.
\end{equation*}
Consequently, we can select $\varepsilon$ and $k$ sufficiently small, independently from $i$ and $j$, so that
\begin{equation}\label{II-est-1}
\sup_{i,j}II_{ij}^{(k)}
\le
\tilde{\varepsilon}\quad\mbox{for any $\tilde{\varepsilon}>0$}.
\end{equation}
Combining \eqref{I-est-1} and \eqref{II-est-1} with \eqref{I+II}, we have
\begin{equation*}\label{Du_i-Du_j}
\lim_{i,j\rightarrow\infty}\int_\Omega |Du_i-Du_j|^{r(x)} \ dx=0,
\end{equation*}
which implies that
$\{Du_i\}_{i\in\N}$ is a Cauchy sequence in $L^{r(\cdot)}(\Omega)$.
In view of \eqref{T_k(u-g)}, we infer
$$
Du_i \rightarrow Du \quad\mbox{in $L^{r(\cdot)}(\Omega)$},
$$
which completes the proof.
\end{proof}

\begin{remark}
For the existence of a limit of approximating solutions, we do not need the assumptions $\diva(D\psi_1,\cdot), \diva(D\psi_2,\cdot), \diva(Dg,\cdot) \in L^1(\Omega) $ in \eqref{psi-g-L1} and \eqref{g-L1}.
Indeed, we can see that there is a limit of approximating solutions in $\mathcal{T}_{g}^{1,p(\cdot)}(\Omega)$ by taking $g=(\psi_1)_+-(\psi_2)_-$.
\end{remark}

\section{Comparison estimates}
\label{comparison}

We consider the weak solution $u\in\mathcal{A}_{\psi_1,\psi_2}$ to the variational inequality
\begin{equation}\label{sola wf}
\integral{\OO}{\mathbf{a}(Du,x)\cdot (D\varphi-Du)}{dx} \ge \integral{\OO}{\mu(\varphi-u)}{dx} \quad \text{for all} \ \ \varphi \in \mathcal{A}_{\psi_1,\psi_2}
\end{equation}
under the assumption
\begin{equation*}
\label{regular} \mu \in L^1(\Omega) \cap W^{-1,p'(\cdot)}(\Omega),
\end{equation*}
where the admissible set $\mathcal{A}_{\psi_1,\psi_2}$ is given by \eqref{admissible set}.

In this section, we establish comparison estimates in $L^1$-sense between the weak solution $u$ of \eqref{sola wf} and the weak solutions of some reference problems.
Throughout this section we assume that $\ga\le n$ and $(p(\cdot),\mathbf{a},\OO)$
is $(\delta,R)$-vanishing. Moreover, we give the two obstacles $\psi_1$ and $\psi_2$ with \eqref{psi-g-L1}.
We only focus on the comparisons near boundary regions of $\Omega$, since the interior case can be derived with the same spirit.
Let $0 < r \le \frac{R_0}{8}$, where $R_0\in(0,1)$ is determined later. We assume the following geometric setting:
\begin{equation}\label{reifen}
B_{8r}^+ \subset \OO_{8r}  \subset B_{8r} \cap \{x_n > -16\delta r\},
\end{equation}
and write
$$p_0 := p(0), \quad p_1:=\inf_{x \in \OO_{8r}} p(x), \quad p_2:=\sup_{x \in \OO_{8r}} p(x),$$
and
\begin{equation*}
\chi_{\{p_0 <2\}} := \left\{\begin{alignedat}{2}
0 & \quad \text{if} \ \ p_0 \ge 2,\\
1 & \quad \text{if} \ \ p_0 < 2.
\end{alignedat}\right.
\end{equation*}
In addition, we denote that for a measurable set $D \subset \bb^n$,
\begin{equation}\label{nu}
|\mu|(D): = \integral{D}{|\mu(x)|}{dx} \quad \text{and} \quad \kappa(D):=|\mu|(D)+|D \cap \OO|,
\end{equation}
and we set
\begin{equation}\label{M-def}
M:= \kappa(\OO)+\int_\Omega |\diva(D\psi_1,x)|\ dx+\int_\Omega |\diva(D\psi_2,x)|\ dx + \integral{\OO}{|Du|}{dx} +1.
\end{equation}

We now introduce the reference problems.
Defining the admissible set
$$
\mathcal{A}_{\psi_1}(\Omega_{8r}):=\mgh{\varphi\in u+W_0^{1,p(\cdot)}(\Omega_{8r}): \varphi\ge\psi_1\ \mbox{a.e. in $\Omega_{8r}$}},
$$
we consider the weak solution $z\in\mathcal{A}_{\psi_1}(\Omega_{8r})$ to the variational inequality
\begin{equation}\label{ref-z}
\int_{\Omega_{8r}}\mathbf{a}(Dz,x)\cdot (D\varphi-Dz)\ dx \ge
\int_{\Omega_{8r}}\mathbf{a}(D\psi_2,x)\cdot (D\varphi-Dz)\ dx
\end{equation}
for all $\varphi\in\mathcal{A}_{\psi_1}(\Omega_{8r})$, where $u$ is the weak solution of the variational inequality \eqref{sola wf}.
We next take into account the following equations, sequentially,
\begin{equation}\label{ref-h}\left\{
\begin{alignedat}{3}
\ddiv \mathbf{a}(Dh,x) &= \ddiv \mathbf{a}(D\psi_1,x) &&\quad \text{in}  \  \OO_{8r}, \\
h &= z &&\quad \text{on} \ \partial \OO_{8r}, \\
\end{alignedat}\right.
\end{equation}
where $z\in\mathcal{A}_{\psi_1}(\Omega_{8r})$ is the weak solution of the variational inequality \eqref{ref-z},
and
\begin{equation}\label{pem2}\left\{
\begin{alignedat}{3}
\ddiv \mathbf{a}(Dw,x) &= 0 &&\quad \text{in} \ \OO_{8r}, \\
w &= h &&\quad \text{on} \ \partial \OO_{8r}, \\
\end{alignedat}\right.
\end{equation}
where $h\in W^{1,p(\cdot)}(\Omega_{8r})$ is the weak solution of the equation \eqref{ref-h}.

We first derive the comparison estimates between \eqref{sola wf} and \eqref{ref-z} as follows:
\begin{lemma}\label{lem-u-z}
Assume that $0 < r \le \frac{R_0}{8}$,
where
$R_0 > 0$ satisfies
\begin{equation}\label{compa1_a}
R_0 \le \min\mgh{\frac{R}{2}, \frac{1}{M}}.
\end{equation}
Let $u\in\mathcal{A}_{\psi_1,\psi_2}$
and $z \in \mathcal{A}_{\psi_1}(\Omega_{8r})$ be the weak solutions to the variational inequalities \eqref{sola wf} and \eqref{ref-z}, respectively.
Then there exists a constant
$c=c(\data)>0$ such that
\begin{equation*}\label{comp-u-z}
\begin{aligned}
&\mint{\OO_{8r}}{|Du-Dz|}{dx} \\
&\quad\le
c \left[\frac{1}{r^{n-1}} \left( \kappa(\Omega_{8r})+\int_{\Omega_{8r}}|\diva(D\psi_2,x)|\ dx \right) \right]^{\frac{1}{p_0-1}} \\
&\qquad
+ \frac{c \chi_{\{p_0 <2\}}}{r^{n-1}} \left( \kappa(\Omega_{8r})+\int_{\Omega_{8r}}|\diva(D\psi_2,x)|\ dx \right)
\left(\mint{\OO_{8r}}{|Du|}{dx}\right)^{2-p_0}.
\end{aligned}
\end{equation*}
\end{lemma}
\begin{proof}
{\em Step 1. Dimensionless estimates.}
Assume that $8r=1$, i.e., $\Omega_{8r}=\Omega_1$. We will prove
\begin{equation}\label{compa1-2}
\mint{\OO_1}{|Du-Dz|}{dx} \le c,
\end{equation}
under the assumption
\begin{equation}
\label{compa1-assump}
P+ \chi_{\{p_1<2\}}P\left(\integral{\OO_1}{|Du|}{dx}\right)^{2-p_1} \le c,
\end{equation}
where the both constants $c$ depend only on $\data$ and
$$
P:=|\mu|(\OO_1)+\int_{\Omega_1}|\diva(D\psi_2,x)|\ dx.
$$

First extend $z$ to $\Omega\setminus\Omega_1$ by $u$. We claim
$$u+T_k(z-u),  u+\Phi_k(z-u) \in\mathcal{A}_{\psi_1,\psi_2} \quad \text{for all} \ k>0,$$
where the truncations $T_k$ and $\Phi_k$ are given in \eqref{truncation} and \eqref{trun}, respectively.
To confirm this, it is enough to show that
$
z\le\psi_2 \ \mbox{a.e. in $\Omega_1$},
$
since
$$\psi_1 \le \min\{u,z\}\le u+T_k(z-u), u+\Phi_k(z-u) \le\max\{u,z\}.$$
Taking $\varphi=\min\{z,\psi_2\}=z-(z-\psi_2)_+\in\mathcal{A}_{\psi_1}(\Omega_1)$ in \eqref{ref-z}, we get
\begin{align*}
\int_{\Omega_1\cap\{z\ge\psi_2\}}(\mathbf{a}(Dz,x)-\mathbf{a}(D\psi_2,x))\cdot(Dz-D\psi_2)\ dx\le0.
\end{align*}
It follows from \eqref{monotonicity} that
\begin{equation}
\begin{aligned}\label{Dz-Dpsi}
&\int_{\Omega_1\cap\{p(x)\ge2\}} |D(z-\psi_2)_+|^{p(x)}\ dx \\
&\qquad +
\int_{\Omega_1\cap\{p(x)<2\}}\big(|Dz|^2+|D\psi_2|^2 \big)^{\frac{p(x)-2}2}|D(z-\psi_2)_+|^2\ dx\le0.
\end{aligned}
\end{equation}
Note that $z\in u+W_0^{1,p(\cdot)}(\Omega_1)$, $u\le\psi_2$ a.e. in $\Omega$,
and $u=0\le\psi_2$ a.e. on $\partial\Omega$.
This means $(z-\psi_2)_+\in W^{1,p(x)}_0(\Omega_1)$,
and hence \eqref{Dz-Dpsi} implies that $(z-\psi_2)_+=0$; namely, $z\le\psi_2$ a.e. in $\Omega_1$.

We now take $\varphi=u+T_k(z-u)$ in \eqref{sola wf}, which follows
\begin{align}\label{u-z}
\int_{\Omega_1}\mathbf{a}(Du,x)\cdot D[T_k(u-z)]\ dx\le
\int_{\Omega_1}\mu T_k(u-z)\ dx.
\end{align}
Since $z+T_k(u-z)\ge\min\{u,z\}\ge\psi_1$, we use $\varphi=z+T_k(u-z)\in\mathcal{A}_{\psi_1}(\Omega_1)$ in \eqref{ref-z}, to have
\begin{align}\label{z-u}
\int_{\Omega_1}\mathbf{a}(Dz,x)\cdot D[T_k(u-z)]\ dx\ge
-\int_{\Omega_1}\diva(D\psi_2,x) T_k(u-z)\ dx.
\end{align}
Subtracting \eqref{z-u} from \eqref{u-z}, we find
\begin{equation}\label{D_k}
\begin{aligned}
&\int_{D_k}(\mathbf{a}(Du,x)-\mathbf{a}(Dz,x))\cdot (Du-Dz)\ dx\\
&\qquad\qquad \le
k\gh{|\mu|(\Omega_1)+\int_{\Omega_1}|\diva(D\psi_2,x)|\ dx},
\end{aligned}
\end{equation}
where $D_k := \OO_1 \cap \{|u-z|\le k\}$.
By testing $\varphi=u+\Phi_k(z-u)$ for \eqref{sola wf} and $\varphi=z+\Phi_k(u-z)$ for \eqref{ref-z},
we similarly obtain
\begin{equation}\label{C_k}
\int_{C_k} \big(\ma(Du,x)-\ma(Dz,x)\big)\cdot(Du-Dz)\ dx \le
|\mu|(\Omega_1)+\int_{\Omega_1}|\diva(D\psi_2,x)|\ dx,
\end{equation}
where $C_k := \OO_1 \cap \{k<|u-z|\le k+1\}$.
Proceeding as in the proof of \cite[Lemma 3.1]{BOP17}, we obtain the estimate \eqref{compa1-2} under \eqref{compa1-assump}.

{\em Step 2. Scalings.}
Let us define
\begin{equation*}
\begin{aligned}\label{scale-maps}
&\tilde{p}(y):=p(8ry), \quad  \tilde{\psi}_1(y) := \frac{\psi_1(8ry)}{8Ar},
\quad \tilde{\psi}_2(y) := \frac{\psi_2(8ry)}{8Ar}, \\
& \tilde{\mu}(y) := \frac{8r\mu(8ry)}{A^{p_0-1}}, \quad \tilde{u}(y) := \frac{u(8ry)}{8Ar}, \quad
 \mbox{and} \quad
\tilde{\mathbf{a}}(\xi,y):=\frac{\mathbf{a}(A\xi,8ry)}{A^{p_0-1}}
\end{aligned}
\end{equation*}
for $y \in \tilde{\OO} := \{y \in \bb^n: 8ry \in \OO \}$ and $\xi \in \bb^n$, where
\begin{equation*}
\begin{aligned}
A:= &\left[\frac{1}{r^{n-1}}\left(\kappa(\Omega_{8r})+\int_{\Omega_{8r}}|\diva(D\psi_2,x)|\ dx\right)\right]^{\frac{1}{p_0-1}} \\
&\ +  \frac{\chi_{\{p_0 <2\}}}{r^{n-1}}\left(\kappa(\Omega_{8r})+\int_{\Omega_{8r}}|\diva(D\psi_2,x)|\ dx\right) \left(\mint{\OO_{8r}}{|Du|}{dx}\right)^{2-p_0}.
\end{aligned}
\end{equation*}
We then check that $\tilde{u}$ is the weak solution to the variational inequality
\begin{equation*}\label{compa1-22}
\int_{\tilde{\Omega}} \ \mathbf{\tilde{a}}(D\tilde{u},y)\cdot(D\varphi-D\tilde{u})\ dy
\ge
\int_{\tilde{\Omega}}\tilde{\mu}(\varphi-\tilde{u})\ dy \quad\mbox{for all $\varphi\in\mathcal{A}_{\tilde{\psi}_1,\tilde{\psi}_2}$},
\end{equation*}
where
$\mathcal{A}_{\tilde{\psi}_1,\tilde{\psi}_2}:=\mgh{\varphi\in W^{1,\tilde{p}(\cdot)}_0(\tilde{\Omega}): \tilde{\psi}_1\le\varphi\le\tilde{\psi}_2 \ \mbox{a.e. in} \ \tilde{\Omega} }$.
Moreover, defining
\begin{align*}\label{scale-maps-1}
\tilde{z}(y) := \frac{z(8ry)}{8Ar} \quad\mbox{for $y\in\tilde{\Omega}_1:=\tilde{\Omega}\cap B_1$},
\end{align*}
we know that $\tilde{z}$ is the weak solution to the variational inequality
\begin{equation*}\label{compa1-23}
\int_{\tilde{\Omega}_1} \ \mathbf{\tilde{a}}(D\tilde{z},y)\cdot(D\varphi-D\tilde{z})\ dy
\ge
\int_{\tilde{\Omega}_1}\mathbf{\tilde{a}}(D\tilde{\psi}_2,x)\cdot(D\varphi-D\tilde{z})\ dy
\end{equation*}
for all $\varphi\in\mathcal{A}_{\tilde{\psi}_1}(\tilde{\Omega}_1)$, where
$\mathcal{A}_{\tilde{\psi}_1}(\tilde{\Omega}_1):=\{\varphi\in \tilde{u}+W^{1,\tilde{p}(\cdot)}_0(\tilde{\Omega}_1): \varphi\ge\tilde{\psi}_1 \ \mbox{a.e. in} \ \tilde{\Omega}_1 \}$.
We can now proceed analogously to the proof of \cite[Lemma 3.1]{BOP17}, to obtain the desired estimate.
\end{proof}

\begin{remark}\label{rem-p-const-2}
If $p^- \ge 2$, then we can replace \eqref{D_k} by
\begin{equation*}
\int_{D_k}|Du-Dz|^{p(x)}\ dx
\le
ck|\mu|(\Omega_1)+c\int_{D_k}|D\psi_2|^{p(x)}\ dx,
\end{equation*}
and \eqref{C_k} by
\begin{equation*}
\int_{C_k}|Du-Dz|^{p(x)}\ dx
\le
c|\mu|(\Omega_1)+c\int_{C_k}|D\psi_2|^{p(x)}\ dx.
\end{equation*}
Proceeding as in the proof of Lemma~\ref{lem-u-z} with two estimates above, instead of \eqref{D_k} and \eqref{C_k}, we conclude
\begin{equation*}
\mint{\OO_{8r}}{|Du-Dz|}{dx} \le c \left[\frac{\kappa(\Omega_{8r})}{r^{n-1}}\right]^{\frac{1}{p_0-1}} +c\gh{\fint_{\Omega_{8r}}|D\psi_2|^{p(x)}\ dx}^{\frac{1}{p_0}}.
\end{equation*}
Thus for $p^- \ge 2$ we can drop the assumption $\diva(D\psi_2,\cdot)\in L^1(\Omega)$ in \eqref{psi-g-L1}.
\end{remark}

\begin{remark}\label{rem-p-const-1}
If $p(\cdot)$ is a constant, then we can deduce a version of Lemma~\ref{lem-u-z} without introducing the measure $\kappa$, see \cite{DM10, DM11, Sch12a}.
However, if $p(\cdot)$ is not a constant, then the presence of $\kappa$ plays a crucial role with log-H\"older continuity of $p(\cdot)$.
\end{remark}

Like similar arguments in the proof of Lemma~\ref{lem-u-z}, we can obtain the following comparison estimates alongside \eqref{ref-z}, \eqref{ref-h}, and \eqref{pem2}.
\begin{lemma}\label{lem-z-h}
Let $z \in \mathcal{A}_{\psi_1}(\Omega_{8r})$ be the weak solution to the variational inequality \eqref{ref-z}, and let $h \in W^{1,p(\cdot)}(\Omega_{8r})$ be the weak solution of the problem \eqref{ref-h}.
Under the assumptions of Lemma~\ref{lem-u-z},
there exists a constant
$c=c(\data)>0$ such that
\begin{align*}
&\mint{\OO_{8r}}{|Dz-Dh|}{dx} \\
&\le
c \left[\frac{1}{r^{n-1}} \left(|\OO_{8r}|+\int_{\Omega_{8r}}|\diva(D\psi_1,x)|\ dx+\int_{\Omega_{8r}}|\diva(D\psi_2,x)|\ dx \right) \right]^{\frac{1}{p_0-1}} \\
&\quad
+  \frac{c \chi_{\{p_0 <2\}}}{r^{n-1}} \left(|\OO_{8r}|+\int_{\Omega_{8r}}|\diva(D\psi_1,x)|\ dx+\int_{\Omega_{8r}}|\diva(D\psi_2,x)|\ dx \right)\\
&\quad\qquad
\times \left(\mint{\OO_{8r}}{|Dz|}{dx}\right)^{2-p_0}.
\end{align*}
\end{lemma}

\begin{lemma}\label{lem-h-w}
Let $h \in W^{1,p(\cdot)}(\Omega_{8r})$ and $w \in W^{1,p(\cdot)}(\Omega_{8r})$ be the weak solutions of the problems \eqref{ref-h} and \eqref{pem2}, respectively.
Under the assumptions of Lemma~\ref{lem-u-z},
there exists a constant
$c=c(\data)>0$ such that
\begin{align*}
&\mint{\OO_{8r}}{|Dh-Dw|}{dx} \\
&\quad\le
c \left[\frac{1}{r^{n-1}} \gh{|\OO_{8r}|+\int_{\Omega_{8r}}|\diva(D\psi_1,x)|\ dx} \right]^{\frac{1}{p_0-1}} \\
&\quad\qquad
+  \frac{c \chi_{\{p_0 <2\}}}{r^{n-1}} \gh{|\OO_{8r}|+\int_{\Omega_{8r}}|\diva(D\psi_1,x)|\ dx} \left(\mint{\OO_{8r}}{|Dh|}{dx}\right)^{2-p_0}.
\end{align*}
\end{lemma}

Next, we provide a higher integrability result of the equation \eqref{pem2}.

\begin{lemma}[{\cite[Lemma 3.3]{BOP17}}]\label{hi1}
Let $M_1>1$ and let $0<r\le \frac{R_0}{8}$. Suppose that $R_0>0$ satisfies
\begin{equation*}\label{hi1_a}
R_0 \le \min\mgh{\frac{R}{2}, \frac{1}{4}, \frac{1}{2M_1}} \quad \text{and} \quad \omega(2R_0) \le \frac{1}{2n}.
\end{equation*}
Then there exists a constant $\tau_0=\tau_0(\data)>0$ such that if $w$ is the weak solution of \eqref{pem2} with
\begin{equation*}\label{hi1_a1}
\integral{\OO_{8r}}{|Dw|}{dx} + 1 \le M_1,
\end{equation*}
then for any $0<\beta\le1$ there exists a constant $c=c(\data,\beta)>0$ so that
\begin{equation*}\label{hi1_r}
\gh{\mint{\OO_{\rho}(x_o)}{|Dw|^{p(x)(1+\sigma)}}{dx}}^{\frac{1}{1+\sigma}} \le c \gh{\mint{\OO_{2\rho}(x_o)}{|Dw|^{p(x)\beta}}{dx}}^{\frac{1}{\beta}} + c,
\end{equation*}
provided $0 < \sigma \le \tau_0$ and $\OO_{2\rho}(x_o) \subset \OO_{8r}$ with $\rho \le 4r$.
\end{lemma}

Now, we derive the universal constant $M_1$ as in Lemma~\ref{hi1}.
Assume that $R_0 > 0$ satisfies \eqref{compa1_a} and put
$$
Q:=\kappa(\Omega)+\int_{\Omega}|\diva(D\psi_1,x)|\ dx+\int_{\Omega}|\diva(D\psi_2,x)|\ dx+1.
$$
From Lemma \ref{lem-u-z}, \ref{lem-z-h} and \ref{lem-h-w}, we directly compute
\begin{equation}\label{M-w-1}
    \begin{aligned}
        \integral{\OO_{8r}}{|Dw|}{dx} +1
         &\le
         c \gh{ \integral{\OO}{|Du|}{dx} + diam(\Omega)^{\frac{n(\ga-2)+1}{\ga-1}} Q^{\frac{1}{\ga-1}}+1}\\
         &=: M_1
      \end{aligned}
\end{equation}
for some $c=c(\data)>0$, where $diam(\OO)$ is the diameter of $\OO$.

Adopting this $M_1$ and applying Lemma~\ref{hi1}, we obtain the higher integrability result for the problem \eqref{pem2} as follows:
\begin{lemma}[{\cite[Lemma 3.5]{BOP17}}]\label{hi3}
Let $0<r\le\frac{R_0}{8}$.
Suppose that $R_0 > 0$ satisfies
\begin{equation*}
R_0 \le \min\mgh{\frac{R}{2},\frac1{M}, \frac{1}{4}, \frac{1}{2M_1}} \quad \text{and} \quad \omega(2R_0) \le \frac{1}{2n},
\end{equation*}
with $M$ as in \eqref{M-def} and $M_1$ as in \eqref{M-w-1}.
Let $w$ be the weak solution of \eqref{pem2}. Then $w$ belongs to $W^{1,p_2}(\OO_{3r})$ with the estimate
\begin{equation*}\label{hi3_r1}
\mint{\OO_{3r}}{|Dw|^{p_2}}{dx} \le c \gh{\mint{\OO_{8r}}{|Dw|}{dx}}^{p_2} + c
\end{equation*}
for some $c=c(\data)>0$.
\end{lemma}

Next, defining a new vector field $\mathbf{B}=\mathbf{B}(\xi,x): \bb^n \times \OO_{8r} \to \bb^n$ by
\begin{equation*}
\mathbf{B}(\xi,x) = \mathbf{a}(\xi,x)|\xi|^{p_2-p(x)},
\end{equation*}
we directly check the following growth and ellipticity conditions:
\begin{equation}\label{growth2}
\left\{\begin{aligned}
&|\xi||D_{\xi}\mathbf{B}(\xi,x)|+|\mathbf{B}(\xi,x)| \le 3 \Lambda_1 |\xi|^{p_2-1},\\
&\frac{\Lambda_2}2 |\xi|^{p_2-2}|\eta|^2 \le  D_{\xi}\mathbf{B}(\xi,x)\eta\cdot\eta
\end{aligned}\right.
\end{equation}
for all $\eta \in \bb^n$, $\xi \in \bb^n \setminus \{0\}$ and $x \in \OO_{8r}$, and for $\Lambda_1$ and $\Lambda_2$ as in \eqref{growth}, whenever
\begin{equation*}
p_2 - p_1 \le \omega(16r) \le \omega(2R_0) \le \min\mgh{1, \frac{\Lambda_2}{2\Lambda_1}},
\end{equation*}
see \cite[Section 4]{BOR16} for details.
We next consider the integral average of $\mathbf{B}(\xi,\cdot)$ on $B_{8r}^+$, denoted by $\bar{\mathbf{B}} = \bar{\mathbf{B}}(\xi)$, as
\begin{equation*}
\bar{\mathbf{B}}(\xi) := \mint{B_{8r}^+}{\mathbf{B}(\xi,x)}{dx}.
\end{equation*}
Then \eqref{growth2} holds with respect to $\bar{\mathbf{B}}(\xi)$. Furthermore we have
\begin{equation*}
\sup_{0<r \le R} \mint{B_r^+}{\theta(\mathbf{a},B_r^+)(x)}{dx} \le 4 \delta,
\end{equation*}
by observing \eqref{small BMO} and
\begin{equation*}
\sup_{\xi \in \bb^n \setminus \{0\}} \frac{\left| \mathbf{B}(\xi,\cdot) - \bar{\mathbf{B}}(\xi) \right|}{|\xi|^{p_2-1}} = \theta(\mathbf{a},B_{8r}^+)(x).
\end{equation*}

We now consider the homogeneous frozen problem
\begin{equation}\label{pem3}\left\{
    \begin{alignedat}{3}
        \ddiv \bar{\mathbf{B}}(Dv) &= 0 &&\quad \text{in} \ \OO_{3r}, \\
        v &= w &&\quad \text{on} \ \partial \OO_{3r}, \\
    \end{alignedat}\right.
\end{equation}
where $w$ is the weak solution of \eqref{pem2}.
Indeed, $w\in W^{1,p_2}(\OO_{3r})$ from Lemma \ref{hi3}, and from the standard energy estimate we directly obtain $v \in  W^{1,p_2}(\OO_{3r})$.

We need to study the comparison estimate between \eqref{pem2} and \eqref{pem3}.
\begin{lemma}[{\cite[Lemma 3.6]{BOP17}}]\label{fcompa1}
Let $0<r\le\frac{R_0}{8}$.
Suppose that $R_0 > 0$ satisfies
\begin{equation*}
R_0 \le \min\mgh{\frac{R}{2},\frac1{M}, \frac{1}{4}, \frac{1}{2M_1}}
\end{equation*}
and
\begin{equation*}
p_2 - p_1 \le \omega(16r) \le \omega(2R_0) \le \min\mgh{\frac1{2n}, \frac{\Lambda_2}{2\Lambda_1}, \frac{\tau_0}{4}},
\end{equation*}
with $M$ as in \eqref{M-def}, $M_1$ as in \eqref{M-w-1} and $\tau_0$ as in Lemma~\ref{hi1}.
Let $w$ be the weak solution of \eqref{pem2}, and let $v$ be the weak solution of \eqref{pem3}.
Then we have the estimate
\begin{equation*}\label{fcompa1_r}
\mint{\OO_{3r}}{|Dw-Dv|^{p_2}}{dx} \le c \delta^{\frac{\tau_0}{4+\tau_0}} \mgh{\gh{\mint{\OO_{8r}}{|Dw|}{dx}}^{p_2} + 1}
\end{equation*}
for some $c=c(\data)>0$.
\end{lemma}

We next consider the following reference problem:
\begin{equation}\label{pem4}\left\{
    \begin{alignedat}{3}
        \ddiv \bar{\mathbf{B}}(D\bar{v}) &= 0 &&\quad \text{in} \ B_{2r}^+, \\
        \bar{v} &= 0 &&\quad\text{on} \ B_{2r} \cap \{x_n = 0\}. \\
    \end{alignedat}\right.
\end{equation}

Then this problem has the following Lipschitz regularity, and we also find a proper comparison estimate between \eqref{pem3} and \eqref{pem4}.
\begin{lemma}[\cite{Lie88}]
\label{fcompa1.5}
Let $\bar{v} \in W^{1,p_2}(B_{2r}^+)$ be a weak solution of \eqref{pem4}.
Then we have $D\bar{v} \in L^{\infty}(B_r^+)$ and
\begin{equation}\label{fcompa1.5_r1}
\Norm{D\bar{v}}_{L^{\infty}(B_r^+)} \le c \mint{B_{2r}^+}{|D\bar{v}|}{dx}
\end{equation}
for some positive constant $c$ depending only on $\data$.
\end{lemma}

\begin{lemma}[{\cite[Lemma 4.6]{BOR16}}]
\label{fcompa2}
For any $0 < \varepsilon <1$, there exists $\delta>0$, depending only on $\data$ and $\varepsilon$, such that if $v \in W^{1,p_2}(\OO_{3r})$ is the weak solution of \eqref{pem3}, then there is a weak solution $\bar{v} \in W^{1,p_2}(B_{2r}^+)$ of \eqref{pem4} satisfying
\begin{equation}\label{fcompa2_r2}
\mint{\OO_{2r}}{|Dv-D\bar{v}|^{p_2}}{dx} \le \varepsilon^{p_2} \mint{\OO_{3r}}{|Dv|^{p_2}}{dx},
\end{equation}
where $\bar{v}$ is extended by zero from $B_{2r}^+$ to $\OO_{2r}$.
\end{lemma}

Combining all the previous estimates on the reference problems, we obtain the final comparison $L^1$-estimates near boundary regions.

\begin{lemma}\label{comparison1}
Let $\rho \ge 1$ and let $0< r \le \frac{R_0}{8}$.
Suppose that $R_0 > 0$ satisfies
\begin{equation}\label{hi1_a-1}
R_0 \le \min\mgh{\frac{R}{2},\frac1{M}, \frac{1}{4}, \frac{1}{2M_1}}
\end{equation}
and
\begin{equation}\label{bc_a}
p_2 - p_1 \le \omega(16r) \le \omega(2R_0) \le \min\mgh{\frac1{2n}, \frac{\Lambda_2}{2\Lambda_1}, \frac{\tau_0}{4}},
\end{equation}
with $M$ as in \eqref{M-def}, $M_1$ as in \eqref{M-w-1} and $\tau_0$ as in Lemma~\ref{hi1}.
Then for any $0 < \varepsilon <1$, there exists a small constant $0<\delta<1$, depending only on
$\data$ and $\varepsilon$, such that if $(p(\cdot),\mathbf{a},\OO)$ is $(\delta,R)$-vanishing and $\OO_{8r}$ forces \eqref{reifen}, and if $u\in\mathcal{A}_{\psi_1,\psi_2}$ is the weak solution of \eqref{sola wf} with
\begin{equation*}\label{u-1}
\mint{\OO_{8r}}{|Du|}{dx} \le \rho
\end{equation*}
and
\begin{equation}\label{compa_assumption}
\left[ \frac{1}{r^{n-1}}\left(\kappa(\Omega_{8r}) +\int_{\Omega_{8r}}\gh{|\diva(D\psi_1,x)| + |\diva(D\psi_2,x)|} dx\right)\right]^{\frac{1}{p_0-1}} \le \delta\rho,
\end{equation}
then there exists a weak solution $\bar{v}$ of \eqref{pem4} such that
\begin{equation*}
\mint{\OO_{2r}}{|Du-D\bar{v}|}{dx} \le \varepsilon\rho \quad \text{and} \quad \Norm{D\bar{v}}_{L^{\infty}(\OO_r)} \le c \rho
\end{equation*}
for some $c=c(\data) > 0$.
\end{lemma}

\begin{remark}\label{rem-p-const-3}
If $p^- \ge 2$, then we can replace the assumption \eqref{compa_assumption} by
\begin{equation*}
\left[ \frac{\kappa(\Omega_{8r})}{r^{n-1}}\right]^{\frac{1}{p_0-1}} +\gh{\fint_{\Omega_{8r}}|D\psi_1|^{p(x)} dx}^{\frac{1}{p_0}} + \gh{\fint_{\Omega_{8r}}|D\psi_2|^{p(x)} dx}^{\frac{1}{p_0}} \le \delta\rho,
\end{equation*}
see Remark~\ref{rem-p-const-2} for details. In this case, we can drop the assumptions $\diva(D\psi_1,\cdot)$, $\diva(D\psi_2,\cdot)\in L^1(\Omega)$ in \eqref{psi-g-L1}.
\end{remark}



\section{Proof of Theorem~\ref{main theorem}}\label{cover}

In this section, we prove our main theorem (Theorem~\ref{main theorem}). We first construct sequences $\{u_i\}$ and $\{\mu_i\}$ satisfying \eqref{sol-def-mu}--\eqref{seq-converge}.
For any $\mu\in\mathcal{M}_b(\Omega)$ we may regard that $\mu$ is defined on $\R^n$ by the zero extension to $\R^n\setminus\Omega$.
Taking $\phi\in C_0^\infty(B_1)$ as the standard mollifier, we define $\phi_i(x):=i^n\phi(ix)$ for $i\in\N$ and $x\in\R^n$.
We now consider
 $$\mu_i:=\mu\ast\phi_i.$$
Then $\mu_i\in C^\infty_c(\R^n)$, in particular, $\mu_i\in W^{-1,p'(\cdot)}(\Omega)\cap L^1(\Omega)$ satisfying \eqref{sol-def-mu} and
\begin{equation}\label{L_1-mu}
\|\mu_i\|_{L^1(\Omega)}\le|\mu|(\Omega).
\end{equation}
Having such a function $\mu_i$, we construct the corresponding weak solution $u_i\in\mathcal{A}_{\psi_1,\psi_2}$ to the variational inequality \eqref{sola wf} with $\mu$ replaced by $\mu_i$  such that
\begin{equation}\label{seq-converge-1}
u_i\rightarrow u \quad \text{in} \  W^{1, r(\cdot)}(\Omega)
\end{equation}
for all continuous functions $r(\cdot)$ on $\overline{\Omega}$ with \eqref{r-range}, as in Lemma~\ref{lem-exist}.

Throughout this section, we suppose that $(p(\cdot),\mathbf{a},\OO)$ is $(\delta,R)$-vanishing.
Moreover, we assume that $R_0>0$
satisfies
\begin{equation}\label{r0-1}
R_0 \le \min \mgh{\frac{R}{2}, \frac{1}{6M_1}, \frac{1}{M+1}} \quad \text{and} \quad
\omega(2R_0) \le \min\mgh{\frac{1}{2n}, \frac{\Lambda_2}{2\Lambda_1},  \frac{\tau_0}{4}},
\end{equation}
where $M$, $M_1$ and $\tau_0$ are given in \eqref{M-def}, \eqref{M-w-1} and Lemma \ref{hi1}, respectively.
According to \eqref{L_1-mu} and \eqref{seq-converge-1}, we see that $R_0$ satisfies \eqref{hi1_a-1} and \eqref{bc_a} with $u$ and $\mu$ replaced by $u_i$ and $\mu_i$, respectively, for sufficiently large $i$.

Our strategy for proving Theorem~\ref{main theorem} is to apply a Vitali type covering lemma (Lemma~\ref{covering}), under the $L^1$-comparison estimates (Lemma~\ref{comparison1}) in Section~\ref{comparison} and Lemma~\ref{covering2} below, to obtain the power decay estimate for upper level sets of $Du$ (see \eqref{ma2} below). Combining this decay estimate, Lemma~\ref{distribution2} and the standard energy $L^1$-estimate (Section~\ref{Standard $L^1$-estimate}), we finally derive the desired regularity estimate \eqref{main_r} in Theorem~\ref{main theorem}.

\subsection{Vitali type covering}\label{Vitali type covering}

For any fixed $\varepsilon \in (0,1)$ and $N>1$, we define
\begin{equation}\label{smallness}
\lambda_0 := \frac{1}{\varepsilon|B_{R_0}|} \mgh{\integral{\OO}{|Du|}{dx} +1} > 1
\end{equation}
and upper-level sets: for $k \in \mathbb{N} \cup \{0\}$,
\begin{equation*}
C_{N,k} := \mgh{x \in \OO : \M(|Du|)(x) > N^{k+1}\lambda_0}
\end{equation*}
and
\begin{equation*}
\begin{aligned}
D_{N,k} &:= \mgh{x \in \OO : \M(|Du|)(x) > N^{k}\lambda_0} \cup \mgh{x \in \OO : \bgh{\M_1(\kappa)(x)}^{\frac{1}{p(x)-1}} > \delta N^{k}\lambda_0}\\
&\hspace{1cm}
\cup \mgh{x \in \OO : \bgh{\M_1(\Psi_1)(x)}^{\frac{1}{p(x)-1}} > \delta N^{k}\lambda_0}\\
&\hspace{2cm}
\cup \mgh{x \in \OO : \bgh{\M_1(\Psi_2)(x)}^{\frac{1}{p(x)-1}} > \delta N^{k}\lambda_0},
\end{aligned}
\end{equation*}
where $\Psi_1(x):=\diva(D\psi_1,x)$ and $\Psi_2(x):=\diva(D\psi_2,x)$ with the conditions \eqref{psi-g-L1}.
Here $\M$ and $\M_1$ are defined by \eqref{max ftn} and \eqref{fractional max ftn}, respectively. Remark that if for $p^- \ge 2$ we instead define
\begin{equation*}
\begin{aligned}
D_{N,k} &:= \mgh{x \in \OO : \M(|Du|)(x) > N^{k}\lambda_0} \cup \mgh{x \in \OO : \bgh{\M_1(\kappa)(x)}^{\frac{1}{p(x)-1}} > \delta N^{k}\lambda_0}\\
&\hspace{1cm}
\cup \mgh{x \in \OO : \bgh{\M\gh{|D\psi_1|^{p(\cdot)}}(x)}^{\frac{1}{p(x)}} > \delta N^{k}\lambda_0}\\
&\hspace{2cm}
\cup \mgh{x \in \OO : \bgh{\M\gh{|D\psi_2|^{p(\cdot)}}(x)}^{\frac{1}{p(x)}} > \delta N^{k}\lambda_0},
\end{aligned}
\end{equation*}
then we can drop the assumptions $\diva(D\psi_1,\cdot)$, $\diva(D\psi_2,\cdot)\in L^1(\Omega)$ in \eqref{psi-g-L1}
(see Remark~\ref{rem-p-const-2} and \ref{rem-p-const-3} for details).

We now verify two assumptions in Lemma~\ref{covering}.
\begin{lemma}\label{covering2}
There exists a constant $N_o=N_o(\data)>1$ such that if $N\ge N_o$, then for any $\varepsilon > 0$,
\begin{enumerate}[\rm(i)]
\item $\displaystyle{|C_{N,k}| \le \frac{\varepsilon}{1000^n}  |B_{R_0}|}$, $k \in \mathbb{N} \cup \{0\}$, and

\item there exists
$0<\delta<1$, depending only on $\data$ and $\varepsilon$, so that
we have $$\OO_{r_0}(y_0) \subset D_{N,k}$$
for all $y_0 \in\Omega$, $r_0 \le \frac{R_0}{1000}$, and $k \in \mathbb{N} \cup \{0\}$,  with
$|C_{N,k} \cap B_{r_0}(y_0)| \ge \varepsilon |B_{r_0}(y_0)|$.
\end{enumerate}
\end{lemma}

\begin{proof}
The proof is almost like that of \cite[Lemma 4.1 and 4.2]{BOP17}.
\end{proof}

\subsection{Standard energy $L^1$-estimate}\label{Standard $L^1$-estimate}
To prove Theorem~\ref{main theorem}, we need the standard energy $L^1$-estimate for the gradient of a solution $u$ given by \eqref{seq-converge-1} (see \eqref{L_1-est-Du} below). To do this, with the function $g\in W^{1,p(\cdot)}_0(\Omega)$ satisfying \eqref{g-L1},
we write
\begin{gather*}
C_k^+:=\mgh{x\in\Omega:k\le |u_i-g|< k+1\ \ \mbox{and}\ \ p(x)\ge2},\\
C_k^-:=\mgh{x\in\Omega:k\le |u_i-g|< k+1\ \ \mbox{and}\ \ p(x)<2} \quad \text{and} \quad C_k := C_k^+ \cup C_k^-
\end{gather*}
for all $k\in\N\cup\{0\}$. Recall the truncation \eqref{trun}.
Taking a comparison function $\varphi=u_i + \Phi_k(g-u_i)$ in the variational inequality \eqref{sola wf} replaced $u$ and $\mu$ by $u_i$ and $\mu_i$, respectively,
we have
\begin{equation*}\label{a-a(u_i-g)}
\begin{aligned}
&\int_{\Omega} \big(\ma(Du_i,x)-\ma(Dg,x)\big)\cdot D[\Phi_k(u_i-g)] \ dx\le
 |\mu_i|(\Omega) + \int_{\Omega} |\diva(Dg,x)| \ dx.
\end{aligned}
\end{equation*}
It follows from \eqref{monotonicity} and \eqref{L_1-mu} that
\begin{equation}
\begin{aligned}\label{C}
&\int_{C_k^+}|Du_i-Dg|^{p(x)} \ dx
+
\int_{C_k^-} \big(|Du_i|^2+|Dg|^2\big)^{\frac{p(x)-2}2}|Du_i-Dg|^2\ dx\\
&\qquad\le
c\underbrace{\left( |\mu|(\Omega) + \int_{\Omega} |\diva(Dg,x)| \ dx \right)}_{=:P}
\end{aligned}
\end{equation}
for all $k\in\N\cup\{0\}.$
Let $s=\frac{2}{3-\ga}$. Note that if $2-\frac1n<\ga<2$, then $s\in(1,\ga)$.
Now we compute
\begin{equation}\label{compa}
\begin{aligned}
&\int_{C_0^-}|Du_i-Dg|^{s}\ dx \\
&\le
\int_{C_0^-}\left(|Du_i|^2+|Dg|^2 \right)^{\frac{(p(x)-2)s}{4}} |Du_i-Dg|^{s} \left(|Du_i|^2+|Dg|^2+1 \right)^{\frac{(2-\ga)s}{4}} dx\\
& \le
cP+\frac12\int_{C_0^-} |Du_i-Dg|^s\ dx+c\int_{C_0^-}\gh{|Dg|+1}\ dx,\end{aligned}
\end{equation}
where we employed  \eqref{C} and Young's inequality with exponents $\frac2{s}$ and $\frac2{(2-\ga)s}$. Indeed, $(2-p^-)s=2-s$.
Then we combine \eqref{C} and \eqref{compa}, to discover
\begin{equation}\label{C_0}
\begin{aligned}
&\int_{C_0}|Du_i-Dg|\ dx\\
&\qquad\le
c\int_{C_0^+}|Du_i-Dg|^{p(x)}\ dx+c \int_{C_0^-}|Du_i-Dg|^s\ dx+c|C_0|\\
&\qquad\le
cP+c\int_\Omega \gh{|Dg|+1}\ dx.
\end{aligned}
\end{equation}
On the other hand, for $k\ge1$ we likewise deduce from \eqref{C} that
\begin{equation}
\begin{aligned}\label{C_k-}
\int_{C_k^-}|Du_i-Dg|^s\ dx
\le
cP^{\frac{s}2}\gh{\int_{C_k^-} \gh{|Du_i|+|Dg|+1}\ dx}^{\frac{2-s}2}.
\end{aligned}
\end{equation}
Then combining \eqref{C} and \eqref{C_k-} yields
\begin{equation*}
\begin{aligned}
\int_{C_k}|Du_i-Dg|\ dx
&\le
\gh{\int_{C_k^+}|Du_i-Dg|^{p(x)}\ dx+|\Omega|}^{\frac1{\ga}}|C_k^+|^{\frac{\ga-1}{\ga}}\\
&\qquad+
\gh{\int_{C_k^-}|Du_i-Dg|^{s}\ dx}^{\frac1{s}}|C_k^-|^{\frac{s-1}{s}}\\
&\le
c(P+|\Omega|)^{\frac1{\ga}}|C_k^+|^{\frac{\ga-1}{\ga}}\\
&\quad+
c \underbrace{\gh{P|C_k^-|^{\ga-1}}^{\frac12} \gh{\int_{C_k^-} \gh{|Du_i|+|Dg|+1}\ dx}^{\frac{2-\ga}{2}}}_{=:(*)}.
\end{aligned}
\end{equation*}
Applying Young's inequality with $\gh{\frac2{\ga},\frac2{2-\ga}}$ to $(*)$, we find
\begin{align*}
\int_{C_k}|Du_i-Dg|\ dx
\le
c(P+|\Omega|)^{\frac1{\ga}}|C_k|^{\frac{\ga-1}{\ga}}
+
c\int_{C_k}\gh{|Dg|+1}\ dx.
\end{align*}
Then we obtain
\begin{align}\label{sum-C-k}
\sum_{k=1}^\infty\int_{C_k}|Du_i-Dg|\ dx
\le
c(P+|\Omega|)^{\frac1{\ga}}\sum_{k=1}^\infty|C_k|^{\frac{\ga-1}{\ga}}
+
c\int_{\Omega}\gh{|Dg|+1}\ dx.
\end{align}
On the other hand, for $\frac{1}{\ga-1} < t < \frac{n}{n-1}$, the definition of $C_k$ means
\begin{align*}
\sum_{k=1}^\infty|C_k|^{\frac{\ga-1}{\ga}}
&\le
\sum_{k=1}^{\infty} \gh{\frac{1}{k}}^{\frac{t(\ga-1)}{\ga}} \gh{\integral{C_k}{|u_i-g|^t}{dx}}^{\frac{\ga-1}{\ga}}.
\end{align*}
Applying H\"{o}lder's inequality and Sobolev's inequality, we have
\begin{equation}
    \begin{aligned}\label{rmk5-1}
        \sum_{k=1}^\infty|C_k|^{\frac{\ga-1}{\ga}}
        &\le \gh{\sum_{k=1}^{\infty} \gh{\frac{1}{k}}^{t(\ga-1)} }^{\frac1{\ga}}
           \gh{\sum_{k=1}^\infty\integral{C_k}{|u_i-g|^t}{dx}}^{\frac{\ga-1}{\ga}}\\
        &\le c \gh{\integral{\OO}{|Du_i-Dg|}{dx}}^{\frac{t(\ga-1)}{\ga}} |\OO|^{\gh{1-\frac{(n-1)t}{n}} \frac{\ga-1}{\ga}}
    \end{aligned}
    \end{equation}
for some $c=c(n,\ga,t)>0$.
Here, we used the fact that $\frac{1}{\ga-1} < t < \frac{n}{n-1}$.
We now combine \eqref{C_0}, \eqref{sum-C-k} and \eqref{rmk5-1} to discover
\begin{align*}
&\integral{\OO}{|Du_i-Dg|}{dx}
=
\int_{C_0}|Du_i-Dg|\ dx+ \sum_{k=1}^\infty\int_{C_k}|Du_i-Dg|\ dx\\\nonumber
&\qquad\qquad\le
c(P+|\Omega|)^{\frac1{\ga}}|\OO|^{\gh{1-\frac{(n-1)t}{n}} \frac{\ga-1}{\ga}}\gh{\integral{\OO}{|Du_i-Dg|}{dx}}^{\frac{t(\ga-1)}{\ga}} \\\nonumber
&\qquad\qquad\qquad+
cP+c\int_{\Omega}\gh{|Dg|+1}\ dx
\end{align*}
for some $c = c(\data,t) > 0$.
After observing that $t < \frac{n}{n-1} \le \frac{\ga}{\ga-1}$ from $\ga \le n$,
we use Young's inequality with $\gh{\frac{\ga}{t(\ga-1)}, \frac{\ga}{\ga-t(\ga-1)}}$, to discover
\begin{equation*}
\begin{aligned}
&\integral{\OO}{|Du_i-Dg|}{dx}\\
&\qquad \le c P  + c {(P+|\Omega|)}^{\frac{1}{\ga-(\ga-1)t}} |\OO|^{\gh{1-\frac{(n-1)t}{n}} \frac{\ga-1}{\ga-t(\ga-1)}}
+
c\int_{\Omega}\gh{|Dg|+1}\ dx.
\end{aligned}
\end{equation*}
Recalling $P$ in \eqref{C} and applying \eqref{seq-converge-1}, we deduce
\begin{align*}
&\integral{\OO}{|Du|}{dx} \\\nonumber
&\quad\le
c \left(|\mu|(\Omega)+|\Omega| + \int_{\Omega} |\diva(Dg,x)| \ dx +\int_{\Omega} |Dg|\ dx\right)  \\\nonumber
&\qquad+ c {\left(|\mu|(\Omega) +|\Omega|+ \int_{\Omega} |\diva(Dg,x)| \ dx \right)}^{\frac{1}{\ga-(\ga-1)t}} |\OO|^{\gh{1-\frac{(n-1)t}{n}} \frac{\ga-1}{\ga-t(\ga-1)}}.
\end{align*}
Selecting $t := \frac{1}{\ga-1} + \alpha$ for small $s$ with $0 < \alpha \le \frac{1}{2}\gh{\frac{n}{n-1} - \frac{1}{\ga-1}} <1$, we therefore conclude the estimate
\begin{equation}
\begin{aligned}\label{L_1-est-Du}
\integral{\OO}{|Du|}{dx}
&\le
c \gh{|\mu|(\OO)+ \int_{\Omega} |\diva(Dg,x)| \ dx + \int_{\Omega} |Dg|\ dx+1} \\
&\qquad
+ c\gh{|\mu|(\OO)+ \int_{\Omega} |\diva(Dg,x)| \ dx +1}^{\frac{1}{(\ga-1)(1-\alpha)}}
\end{aligned}
\end{equation}
for some $c=c(\data,\alpha,\OO)>0$.

\begin{remark}\label{rem-p-const-3.5}
We note that $c$ in \eqref{L_1-est-Du} goes to $+\infty$ when $\alpha$ goes to zero, since the constant $c$ in \eqref{rmk5-1} blows up as $\alpha \searrow 0$.
On the other hand, if $p(\cdot)$ is a constant, then we can deduce from the normalization property that
\begin{align}\label{L_1-const}
\int_\Omega |Du| \ dx\le
c\int_{\Omega} \mathcal{M}_1(\mu)^{\frac1{p-1}}\ dx +
c\int_{\Omega} \mathcal{M}_1(G)^{\frac1{p-1}}\ dx +
c\int_{\Omega} |Dg|\ dx
\end{align}
for some $c=c(n,\Lambda_1,\Lambda_2,p)>0$,
where $G(\cdot):=\diva(Dg,\cdot)$.
Indeed, this estimate can be derived by a similar way with the proof of Lemma~\ref{lem-u-z}.
\end{remark}

\begin{remark}\label{rem-p-const-4}
If $p^- \ge 2$, then we can replace the estimate \eqref{C} by
\begin{equation*}
\int_{C_k}|Du_i-Dg|^{p(x)} \ dx \le
c |\mu|(\Omega) + c\int_{\Omega} |Dg|^{p(x)} \ dx.
\end{equation*}
In this case, we infer
\begin{equation*}
\begin{aligned}
\integral{\OO}{|Du|}{dx}
&\le
c \gh{|\mu|(\OO)+ \int_{\Omega} |Dg|^{p(x)} \ dx + 1} \\
&\qquad
+ c\gh{|\mu|(\OO)+ \int_{\Omega} |Dg|^{p(x)} \ dx +1}^{\frac{1}{(\ga-1)(1-\alpha)}}
\end{aligned}
\end{equation*}
for some $c=c(\data,\alpha,\OO)>0$. Thus for $p^- \ge 2$, we can drop the assumption $\diva(Dg,\cdot)\in L^1(\Omega)$  in \eqref{g-L1}.
\end{remark}

\subsection{Proof of Theorem~\ref{main theorem}}\label{proof of main theorem}

Let $\varepsilon>0$. Taking $N_o$ and $\delta$ given in Lemma \ref{covering2}, we apply Lemma \ref{covering} to have
\begin{equation*}\label{ma1}
|C_{N_o,k}| \le  \left(\frac{80}{7}\right)^n\varepsilon |D_{N_o,k}| =:  \varepsilon_1 |D_{N_o,k}| \quad \text{for all} \ k \in \mathbb{N} \cup \{0\}.
\end{equation*}
An iteration argument and Fubini's theorem yield a power decay estimate for upper level sets of $\M(|Du|)$ as follows:
\begin{equation}\label{ma2}
    \begin{aligned}
        S &:= \sum_{k=1}^{\infty} N_o^{qk} \norm{\mgh{x \in \OO : \M(|Du|)(x) > N^{k}\lambda_0}}\\
        &\le |\OO| \sum_{k=1}^{\infty} \gh{N_o^q \varepsilon_1}^k\\
        &\quad + \sum_{i=1}^{\infty} \gh{N_o^q \varepsilon_1}^i \sum_{k=i}^{\infty} N_o^{q(k-i)}  \norm{\mgh{x \in \OO : \bgh{\M_1(\kappa)(x)}^{\frac{1}{p(x)-1}} > \delta N_o^{k-i}\lambda_0}}\\
        &\quad + \sum_{i=1}^{\infty} \gh{N_o^q \varepsilon_1}^i \sum_{k=i}^{\infty} N_o^{q(k-i)}  \norm{\mgh{x \in \OO : \bgh{\M_1(\Psi_1)(x)}^{\frac{1}{p(x)-1}} > \delta N_o^{k-i}\lambda_0}}\\
        &\quad + \sum_{i=1}^{\infty} \gh{N_o^q \varepsilon_1}^i \sum_{k=i}^{\infty} N_o^{q(k-i)}  \norm{\mgh{x \in \OO : \bgh{\M_1(\Psi_2)(x)}^{\frac{1}{p(x)-1}} > \delta N_o^{k-i}\lambda_0}}.
    \end{aligned}
\end{equation}
It follows from Lemma~\ref{distribution2} that
\begin{equation}\label{ma3}
\begin{aligned}
S &\le 4|\OO| + \frac{c}{\lambda_0^q}\integral{\OO}{\M_1(\kappa)^{\frac{q}{p(x)-1}}}{dx}+\frac{c}{\lambda_0^q}\integral{\OO}{\M_1(\Psi_1)^{\frac{q}{p(x)-1}}}{dx}\\
&\qquad+\frac{c}{\lambda_0^q}\integral{\OO}{\M_1(\Psi_2)^{\frac{q}{p(x)-1}}}{dx},
\end{aligned}
\end{equation}
where we selected $\varepsilon_1$ so that $N_o^q \varepsilon_1 \le \frac{1}{2}$, and then $\delta = \delta(\data,q) > 0$ is determined.
Meanwhile, Lemma~\ref{distribution2} implies
\begin{equation}\label{ma3.5}
    \begin{aligned}
        \integral{\OO}{|Du|^q}{dx} &\le \integral{\OO}{\M(|Du|)^q}{dx} \le c \lambda_0^q \gh{|\OO|+S}.
    \end{aligned}
\end{equation}
Inserting \eqref{ma3} into \eqref{ma3.5} and recalling $\eqref{r0-1}$ and \eqref{smallness}, we find
\begin{equation}\label{ma3.5-2}
    \begin{aligned}
        \integral{\OO}{|Du|^q}{dx}
        &\le c \gh{\frac1{R_0}}^{(n+1)q}
         + c\integral{\OO}{\M_1(\kappa)^{\frac{q}{p(x)-1}}}{dx}\\
        &\qquad+c\int_\Omega\mathcal{M}_1(\Psi_1)^{\frac{q}{p(x)-1}}\ dx
        +c\int_\Omega\mathcal{M}_1(\Psi_2)^{\frac{q}{p(x)-1}}\ dx
    \end{aligned}
\end{equation}
for some $c = c(\data,q,\Omega)>0$.
Using \eqref{L_1-est-Du} and \eqref{r0-1}, we can choose $R_0>0$ satisfying
\begin{align}\label{R_0-est}
\frac{1}{R_0} &\le \frac{c}{R} \mgh{ V+V^{\frac1{(\ga-1)(1-\alpha)}}+1}
\end{align}
for some $c = c(\data,\omega(\cdot),\alpha,\OO)>0$ and for some $R<1$, where $V$ is given in \eqref{VV}.
Furthermore, recalling \eqref{nu}, we compute for $x \in \OO$,
\begin{equation}\label{ma5}
    \begin{aligned}
        \M_1(\kappa)(x) &:= \sup_{r>0} \frac{r \nu(B_r(x))}{|B_r(x)|}
        \le \sup_{r>0} \frac{r |\mu|(B_r(x))}{|B_r(x)|} + \sup_{r>0} \frac{r |B_r(x) \cap \OO|}{|B_r(x)|}\\
        &\le \M_1(\mu)(x) + c(n) |\OO|^{\frac{1}{n}}.
    \end{aligned}
\end{equation}
Finally, we employ \eqref{ma3.5-2}, \eqref{R_0-est} and \eqref{ma5} to obtain the desired estimate \eqref{main_r}, which proves Theorem \ref{main theorem}.


\begin{remark}\label{main rk3}
If $p(\cdot)$ is a constant, then it follows from Remark~\ref{rem-p-const-1} and \eqref{ma3.5-2} that
\begin{equation*}
    \begin{aligned}
        \integral{\OO}{|Du|^q}{dx}
        &\le c \frac{|\OO|}{R^{nq}} \gh{\integral{\OO}{|Du|}{dx} +1}^q + c\integral{\OO}{\M_1(\mu)^{\frac{q}{p-1}}}{dx}\\
        &\qquad+c\int_\Omega\mathcal{M}_1(\Psi_1)^{\frac{q}{p-1}}\ dx
        +c\int_\Omega\mathcal{M}_1(\Psi_2)^{\frac{q}{p-1}}\ dx
    \end{aligned}
\end{equation*}
for some $c=c(n,\Lambda_1,\Lambda_2,p,q)>0$.
This estimate along with \eqref{L_1-const} leads to
\begin{align*}
\integral{\OO}{|Du|^q}{dx}
&\le c \integral{\OO}{\M_1(\mu)^{\frac{q}{p-1}}}{dx}
+c\integral{\OO}{\M_1(\Psi_1)^{\frac{q}{p-1}}}{dx}
+c\integral{\OO}{\M_1(\Psi_2)^{\frac{q}{p-1}}}{dx}\\
&\qquad+c\gh{\integral{\OO}{\M_1(G)^{\frac{1}{p-1}}}{dx}
+\integral{\OO}{|Dg|}{dx}}^q+c
\end{align*}
for some $c=c(n,\Lambda_1,\Lambda_2,p,q,R,\OO)>0$, where $G(\cdot):=\diva(Dg,\cdot)$.
\end{remark}




\end{document}